\documentclass[11pt]{article}

\usepackage[T1]{fontenc}
\usepackage[utf8]{inputenc}
\usepackage{amsmath,amssymb,amsthm}
\usepackage{geometry}
\usepackage[numbers,sort&compress]{natbib}
\usepackage[hidelinks]{hyperref}

\newtheorem{theorem}{Theorem}[section]
\newtheorem{proposition}[theorem]{Proposition}
\newtheorem{lemma}[theorem]{Lemma}

\begin{document}

\title{Information geometry of emergent symmetry quotients and their weak unfoldings}

\author{
Arnaud Coatanhay\thanks{\texttt{arnaud.coatanhay@ensta.fr}}
\and
Ang\'elique Dr\'emeau\thanks{\texttt{angelique.dremeau@ensta.fr}}
}

\date{}

\maketitle

\begin{center}
\small
Lab-STICC, UMR CNRS 6285, ENSTA, Institut Polytechnique de Paris\\
2 rue Fran\c{c}ois Verny, 29806 Brest Cedex 9, France
\end{center}

\begin{abstract}
A regular observed statistical model may converge to a limit in which a
previously identifiable signed parameter becomes identifiable only modulo a
reflection. We study the local information geometry of this transition.
For a twice differentiable Hellinger embedding with an exact limiting
reflection, the observed displacement is forced into the two-jet form
\(\varepsilon\lambda J_-+\lambda^2J_+/2\), up to higher-order terms.
The mixed jet restores the sign away from the symmetric face, whereas the even
jet is the first tangent inherited by the quotient. After nuisance elimination, a positive Gram determinant yields a
nondegenerate cross-cap two-jet.  The associated local asymptotic theory has
three regimes governed by \(\tau_n=\sqrt n\,\varepsilon_n^2\): regular
signed LAN, a critical curved Gaussian subexperiment, and a quotient regime
with the \(n^{-1/4}\) signed scale. We prove that the same parabolic critical
experiment persists for predictive likelihoods along a single stationary
dependent trajectory.
The limiting quotient has a regular Fisher metric in the invariant
coordinate, while its pullback degenerates in the signed coordinate. For a
solvable CIR--OU benchmark motivated by coherent sea-clutter observations, we
derive the quotient Fisher metric and curvature explicitly and show that the
curvature is strictly negative.  We also determine the restricted holonomy of
the full Amari family: \(\operatorname{Hol}_0(\nabla^{(a)})=SO(2)\) for
\(a=0\), whereas \(\operatorname{Hol}_0(\nabla^{(a)})=GL^+(2,\mathbb R)\)
for \(a\neq0\). The results separate the intrinsic geometry of the limiting quotient from the
transverse geometry of its weak unfolding.
\end{abstract}


\section{Introduction}
\label{sec:introduction}

A statistical parameter can become non-identifiable not because the original
model was singular, but because a symmetry emerges after partial observation
in a limiting regime.  We study such a transition: a regular observed family
approaches an experiment in which a signed parameter is identifiable only
modulo reflection.  The central questions are how the regular family meets the
identifiable quotient, what local asymptotic experiment resolves the crossover,
and which geometric data belong to the quotient rather than to its unfolding.

Information loss under a statistic or observation channel is naturally measured
through Fisher contraction and sufficiency, with recent work also quantifying
approximate sufficiency \cite{AyJostLeSchwachhofer2017,YamaguchiNozawa2024}.
Here the issue is not only the amount of contraction: partial observation changes
the identifiable symmetry type itself in the limit.

Singular Fisher information and non-standard rates have substantial existing
theories.  Rank-deficient statistical models admit reduced Fisher structures
\cite{JostLeSchwachhofer2017,Le2020}; broader singular learning theory treats
non-identifiable models whose Fisher matrices lose rank \cite{Watanabe2009}.
Singular information can change likelihood asymptotics and convergence rates
\cite{RotnitzkyCoxBottaiRobins2000,Drton2009,HallinLey2014,EkvallBottai2022}.
Quotient identifiability and local Gaussian asymptotics also arise in quantum
Markov models, where identifiable parameters are group orbits and may form
stratified orbifold-like spaces \cite{GutaKiukas2017,GirottiKiukasGuta2026}.
The distinction here is that the quotient is not present from the start:
it emerges as the limit of a regular observed experiment.

A solvable realization is supplied by the partially observed CIR--OU
texture--coherence model motivated by coherent sea clutter.  Its process-level
geometry and observability are treated separately in
\cite{CoatanhayDremeau2026}; here it serves only as a nontrivial statistical
experiment for which the required jets can be computed explicitly.  With
$\varepsilon=\alpha^{-1/2}$ denoting the weak-texture scale and $\lambda$ a
signed coupling, the strict limit satisfies
\begin{equation}
  P^Y_{0,\lambda}=P^Y_{0,-\lambda}.
  \label{eq:intro-reflection}
\end{equation}
Thus the ordinary $\lambda$ score vanishes at the fixed point, whereas the
invariant coordinate
\begin{equation}
  \mu=\frac{\lambda^2}{2}
  \label{eq:intro-mu}
\end{equation}
retains strictly positive efficient Fisher information.  The limiting model is
therefore singular in the signed parametrization but regular on its identifiable
quotient.

The first main result is an observed Hellinger normal form.  If
$\Phi(\varepsilon,\lambda)=2\sqrt{p_{\varepsilon,\lambda}}$ is twice
Fr\'echet differentiable and
$\Phi(0,\lambda)=\Phi(0,-\lambda)$, then
\begin{equation}
\begin{split}
  2\sqrt{p_{\varepsilon,\lambda}}
  -2\sqrt{p_{\varepsilon,0}}
  ={}&
  \sqrt{p_0}
  \left(
    \varepsilon\lambda J_-
    +\frac{\lambda^2}{2}J_+
  \right)
  +o_{L^2}(\varepsilon|\lambda|+\lambda^2).
\end{split}
  \label{eq:intro-normal-form}
\end{equation}
The mixed jet $J_-$ restores orientation away from the symmetric face; the even
jet $J_+$ is the first tangent inherited by the quotient.  A hidden involution
with an anti-invariant first perturbation provides one sufficient latent
mechanism, but the theorem itself is stated at the observed level.
After nuisance projection, linear independence of
$J_-^\perp$ and $J_+^\perp$ produces the effective two-jet
\begin{equation}
  (\varepsilon,\lambda)
  \longmapsto
  \left(
    \varepsilon,\varepsilon\lambda,\frac{\lambda^2}{2}
  \right),
  \label{eq:intro-crosscap}
\end{equation}
which is the classical cross-cap normal form at second order
\cite{GolubitskyGuillemin1973}.  We use the terminology only as a local
organizer; related singularity-theoretic methods already occur in information
geometry \cite{KabataMatsumotoUchidaUeki2025}.

The same two jets determine the local asymptotic resolution.  In the
triangular array the relevant map is
\begin{equation}
  \sqrt n
  \left(
    \varepsilon_n\lambda_n,\frac{\lambda_n^2}{2}
  \right),
  \qquad
  \tau_n=\sqrt n\,\varepsilon_n^2.
  \label{eq:intro-resolution}
\end{equation}
Standard Hellinger/LAN theory provides the ambient Gaussian shift framework
\cite{LeCamYang2000,vanDerVaart1998}; the geometry of the resolution map then
produces three restrictions.  When $\tau_n\to\infty$, one recovers regular
signed LAN.  When $\tau_n\to0$, the regular local coordinate is $\mu$ and the
signed scale is $n^{-1/4}$.  At the critical scale
$\tau_n\to\tau\in(0,\infty)$, the limit is a curved Gaussian subexperiment
with mean curve
\begin{equation}
  h_\zeta^\perp
  =
  \zeta J_-^\perp
  +\frac{\zeta^2}{2}J_+^\perp.
  \label{eq:intro-critical-parabola}
\end{equation}
Its nonzero statistical curvature is controlled by the efficient Gram
determinant in the sense of Efron's curved-family geometry
\cite{Efron1975}.
For the CIR--OU benchmark,
$\tau_n=\sqrt n/\alpha_n$, so the critical layer is
$\alpha_n\asymp\sqrt n$.

The crossover is not restricted to independent repetitions.  Predictive
factorization turns the odd and even jets into martingale differences.  Under a
uniform remainder estimate and a predictable Gram limit, the same parabolic
Gaussian experiment survives along one stationary trajectory.  Memory changes
the metric of the critical plane, not its two-jet form or its $\sqrt n$ scale.
This sits alongside, but is distinct from, information-geometric treatments of
Markov kernels themselves \cite{WolferWatanabe2021}; the probabilistic input is
a martingale-array limit theorem \cite{HallHeyde1980}.

The second geometric layer belongs to the exact quotient.  The reflection
relation fails ordinary Godement regularity at the fixed locus, while the
coarse identifiable space is the manifold with boundary
$Q=\{(\beta,\mu):\mu\ge0\}$.  Invariant theory explains why $\mu$ is the
first normal coordinate, and the Fisher metric extends non-degenerately to the
boundary.  This quotient geometry is not determined by the cross-cap two-jet:
the first quotient tangent occurs at order $\lambda^2$, intrinsic Fisher
curvature requires the order-$\lambda^4$ jet, and the complete holonomy
certificate reaches order $\lambda^6$.

For the solvable benchmark we compute these higher even jets explicitly.  The
one-sided boundary value of the Fisher Gaussian curvature is strictly negative
and independent of the slow sampling parameter after normalization.  On the
regular interior, the full Amari family has restricted holonomy
\begin{equation}
  \operatorname{Hol}_0(\nabla^{(a)})
  =
  \begin{cases}
    SO(2),&a=0,\\
    GL^+(2,\mathbb R),&a\neq0.
  \end{cases}
  \label{eq:intro-holonomy}
\end{equation}
The algebraic certificate for the non-metric connections is reproduced in the
Supplementary Material.

The paper is organized to keep these two layers separate.  Section~2 introduces
the solvable observed experiment.  Sections~3--4 identify the reflection
quotient and its weak unfolding.  Sections~5--6 derive the independent and
predictive asymptotic crossovers.  Section~7 proves that quotient and unfolding
jets contain different information.  Section~8 develops the intrinsic
Fisher--Amari geometry.  Technical asymptotic and model-specific calculations
are kept in the appendices, while the longest algebraic and uniform-remainder
certificates are placed in the Supplementary Material.

\section{A solvable partially observed experiment}
\label{sec:benchmark}

This section isolates, in a concrete model, the statistical phenomenon
treated abstractly later.
We retain only the ingredients needed for the information-geometric analysis.
The process-level geometry and observability properties of the broader
texture--coherence class are discussed separately in
\cite{CoatanhayDremeau2026}; the present paper concerns the family of observed
probability laws.

\subsection{Texture--coherence dynamics and normalized coupling}
\label{subsec:benchmark-dynamics}

We start from Field's CIR--OU benchmark
\cite{FieldHaykin2008,Field2008} and embed it into a one-parameter family in
which the fast relaxation rate is allowed to depend on the texture.
The constant-rate model is recovered at $\lambda=0$; the state-dependent
family introduced below is the statistical deformation studied in this paper.
The positive texture process satisfies
\begin{equation}
  dx_t
  =
  A(1-x_t)\,dt
  +
  \sqrt{\frac{2A}{\alpha}x_t}\,dW^x_t,
  \qquad
  A>0,\quad \alpha>0,
  \label{eq:cir-texture}
\end{equation}
and the complex coherence variable
\(
  \gamma_t=\gamma_t^{R}+i\gamma_t^{I}
\)
has two independent real quadratures whose relaxation rate is allowed to depend
on the texture,
\begin{align}
  d\gamma_t^{R}
  &=
  -\frac{B_\lambda(r_t)}{2}\gamma_t^{R}\,dt
  +
  \sqrt{\frac{B_\lambda(r_t)}{2}}\,dW_t^{R},
  \label{eq:coupled-ou-R}
  \\
  d\gamma_t^{I}
  &=
  -\frac{B_\lambda(r_t)}{2}\gamma_t^{I}\,dt
  +
  \sqrt{\frac{B_\lambda(r_t)}{2}}\,dW_t^{I}.
  \label{eq:coupled-ou-I}
\end{align}
The Brownian drivers are independent.
The coherent observation is
\begin{equation}
  \Psi_t=\sqrt{x_t}\,\gamma_t.
  \label{eq:coherent-observation}
\end{equation}

For later use, introduce the Lamperti coordinate
\begin{equation}
  r=\sqrt{\frac{2\alpha x}{A}}
  \label{eq:lamperti-r}
\end{equation}
and denote by $\pi_\alpha$ the stationary law of $r$ induced by the CIR
stationary distribution.  Let
\begin{equation}
  z_\alpha(r)
  =
  \frac{r-\mathbb E_{\pi_\alpha}[r]}
       {\sqrt{\operatorname{Var}_{\pi_\alpha}(r)}}
  \label{eq:standardized-r}
\end{equation}
be its stationary standardized version.
We consider the normalized exponential family
\begin{equation}
  B_{\lambda}^{(\alpha)}(r)
  =
  B_0\,
  \frac{\exp\!\bigl(\lambda z_\alpha(r)\bigr)}
       {\mathbb E_{\pi_\alpha}
        [\exp(\lambda z_\alpha)]},
  \qquad B_0>0.
  \label{eq:normalized-coupling}
\end{equation}
By construction,
\begin{equation}
  \mathbb E_{\pi_\alpha}
  [B_{\lambda}^{(\alpha)}]
  =
  B_0.
  \label{eq:normalized-rate}
\end{equation}
At $\lambda=0$ one recovers the standard Field model with constant fast rate
$B_0$.

The normalization in \eqref{eq:normalized-coupling} separates the mean fast
relaxation rate from its state dependence.
For each fixed texture state, the OU generators
\eqref{eq:coupled-ou-R}--\eqref{eq:coupled-ou-I} have invariant law
\(
  N(0,1/2)
\)
for each quadrature, independently of the value of the positive rate.
Consequently, the product measure
\begin{equation}
  x\sim \Gamma(\alpha,\text{rate }\alpha),
  \qquad
  \gamma^{R},\gamma^{I}\stackrel{\mathrm{iid}}{\sim}N(0,1/2),
  \label{eq:stationary-law}
\end{equation}
is invariant for every $\lambda$; under the stationary experiment used below,
texture and coherence are therefore independent at a fixed time.
In particular, the one-time stationary law of $\Psi=\sqrt{x}\gamma$ is
independent of $\lambda$.
All information on the coupling is therefore temporal.

\subsection{The observed transition experiment}
\label{subsec:observed-transition}

Fix a sampling interval $\Delta t>0$ and write
\begin{equation}
  Y_k
  =
  \begin{pmatrix}
    \Re\Psi_{k\Delta t}\\
    \Im\Psi_{k\Delta t}
  \end{pmatrix}
  \in\mathbb R^2.
  \label{eq:sampled-observation}
\end{equation}
The basic finite-dimensional experiment used below is the stationary observed
two-time law of
\(
  (Y_0,Y_1)
\).
Equivalently, because the stationary one-time law of $Y_0$ is independent of
$(\beta,\lambda)$ in the normalized family, parameter scores may be computed
from the conditional two-time density of $Y_1$ given $Y_0$.
No Markov property of the marginal process $(Y_k)$ is assumed here.
Independent repetitions of the stationary pair experiment will be used in the
local triangular-array analysis, whereas a single long stationary sequence
\(
  (Y_k)_{k\geq0}
\)
will be considered in the predictive extension.

We use
\begin{equation}
  \beta=\log B_0
  \label{eq:beta-nuisance}
\end{equation}
as the fast-rate coordinate and treat it as a nuisance parameter when the
coupling $\lambda$ is the parameter of interest.
Two dimensionless time scales occur repeatedly:
\begin{equation}
  d=A\Delta t,
  \qquad
  c=\frac{B_0\Delta t}{2},
  \qquad
  \rho=e^{-c}.
  \label{eq:dc-rho}
\end{equation}
The parameter $d$ measures the texture evolution across one sampling interval,
while $\rho$ is the one-step correlation of a constant-rate OU quadrature.

The experiment is partially observed in two distinct senses.
First, the instantaneous map
\(
  (x,\gamma)\mapsto\Psi=\sqrt{x}\gamma
\)
is non-injective.
Second, even after a finite number of coherent samples, the texture path
between observation times remains latent.
The likelihood considered here is therefore the marginal likelihood of the
observed variables, with the texture integrated out.

\subsection{Gaussian-texture limit and emergent reflection symmetry}
\label{subsec:gaussian-limit}

Set
\begin{equation}
  \varepsilon=\alpha^{-1/2}.
  \label{eq:epsilon-alpha}
\end{equation}
At stationarity, $x\to1$ in probability as $\alpha\to\infty$.
More precisely, the centered fluctuation
\begin{equation}
  Z_t^{(\alpha)}
  =
  \sqrt{\alpha}\,(x_t-1)
  \label{eq:standardized-texture}
\end{equation}
converges to the stationary Gaussian OU process
\begin{equation}
  dZ_t=-A Z_t\,dt+\sqrt{2A}\,dW_t^x,
  \qquad
  Z_t\sim N(0,1),
  \label{eq:texture-ou-limit}
\end{equation}
and the standardized Lamperti variable $z_\alpha(r_t)$ has the same limit.
We denote this limiting process by $Z$.
Its path law is invariant under
\(
  Z\mapsto-Z
\).
In the same limit, the normalized coupling
\eqref{eq:normalized-coupling} becomes
\begin{equation}
  B_\lambda(Z)
  =
  B_0
  \exp\!\left(
    \lambda Z-\frac{\lambda^2}{2}
  \right),
  \label{eq:gaussian-coupling}
\end{equation}
and therefore satisfies the pointwise identity
\begin{equation}
  B_{-\lambda}(-Z)=B_\lambda(Z).
  \label{eq:rate-reflection}
\end{equation}

At $\varepsilon=0$ the observation amplitude no longer contains the latent
Gaussian fluctuation: $\Psi=\gamma$.
Hence changing
\(
  (Z,\lambda)
\)
into
\(
  (-Z,-\lambda)
\)
leaves the conditional fast dynamics unchanged and does not alter the observed
variable.
The sign symmetry of the latent Gaussian path therefore survives
marginalization.

\begin{proposition}[Limiting observed reflection symmetry]
\label{prop:benchmark-reflection}
In the strict Gaussian-texture limit,
the observed finite-dimensional distributions satisfy
\begin{equation}
  P^{Y}_{0,\lambda}
  =
  P^{Y}_{0,-\lambda}.
  \label{eq:benchmark-reflection}
\end{equation}
In particular, whenever the observed density is differentiable in $\lambda$,
\begin{equation}
  \left.
  \partial_\lambda
  \log p^Y_{0,\lambda}
  \right|_{\lambda=0}
  =
  0
  \qquad
  P^Y_{0,0}\text{-a.s.}
  \label{eq:lambda-score-zero}
\end{equation}
and the Fisher metric written in the signed coordinate $\lambda$ loses rank at
the fixed point.
\end{proposition}

\begin{proof}
For a latent path $Z$, the conditional law of the fast process under
parameter $\lambda$ depends on $Z$ through $B_\lambda(Z)$.
Equation \eqref{eq:rate-reflection} shows that this conditional law is
unchanged by
\(
  (Z,\lambda)\mapsto(-Z,-\lambda)
\).
Since $Z$ and $-Z$ have the same path law and the limiting observation no
longer depends on $Z$ through its amplitude, integration over the latent path
gives \eqref{eq:benchmark-reflection}.
Differentiating the identity at $\lambda=0$ yields
\eqref{eq:lambda-score-zero}.
\end{proof}

The essential point is that the reflection in
\eqref{eq:benchmark-reflection} is an \emph{observed statistical symmetry}.
For finite $\alpha$, the texture fluctuation also enters the amplitude
$\sqrt{x}$ in \eqref{eq:coherent-observation}; the transformation
$Z\mapsto-Z$ is then no longer invisible to the observer, and the sign
symmetry is weakly broken.

\subsection{The identifiable quotient coordinate}
\label{subsec:quotient-coordinate}

The symmetry \eqref{eq:benchmark-reflection} suggests replacing the signed
coordinate by the invariant coordinate
\begin{equation}
  \mu=\frac{\lambda^2}{2},
  \qquad
  \mu\geq0.
  \label{eq:mu-def}
\end{equation}
We now show that the vanishing first-order score in $\lambda$ does not imply
loss of the first invariant direction.

Let
\(
  p_{\beta,\mu}(y_1\mid y_0)
\)
denote the conditional density associated with the stationary observed
two-time law in the strict Gaussian-texture limit.
At $\lambda=0$ this is the ordinary Gaussian OU transition density; for
$\mu>0$ the notation refers only to the two-time conditional law and does not
assert that the marginal observed process is Markov.
Since the stationary density of $Y_0$ is parameter independent, its score is
also the score of the stationary pair experiment.
At $\mu=0$, write
\begin{equation}
  q
  =
  \partial_\beta\log p_{\beta,0},
  \qquad
  h
  =
  \partial_{\beta\beta}^2\log p_{\beta,0},
  \qquad
  k=h-q+q^2.
  \label{eq:qhk-def}
\end{equation}
A direct second-order marginal-likelihood calculation gives
\begin{equation}
  S_\mu
  :=
  \left.
  \partial_\mu\log p_{\beta,\mu}
  \right|_{\mu=0}
  =
  V(d)\,k,
  \qquad
  V(d)
  =
  \frac{2(d-1+e^{-d})}{d^2}.
  \label{eq:mu-score}
\end{equation}
The explicit calculation is collected in Appendix~\ref{app:cir-ou-formulas}.

For the Gaussian fast transition, define
\begin{equation}
  I(c)
  =
  \mathbb E[q^2]
  =
  \frac{
    2c^2\rho^2(1+\rho^2)
  }{
    (1-\rho^2)^2
  },
  \label{eq:Ic}
\end{equation}
and
\begin{equation}
  M_2(c)
  =
  \mathbb E[k^2]
  =
  \frac{
    2c^4\rho^2
    \bigl(
      7\rho^6+19\rho^4+5\rho^2+1
    \bigr)
  }{
    (1-\rho^2)^4
  }.
  \label{eq:M2c}
\end{equation}
The mixed moment is
\begin{equation}
  \mathbb E[qk]=-c\,I(c).
  \label{eq:qk-cross}
\end{equation}
Hence the Fisher matrix of the limiting quotient experiment at $\mu=0$ is
\begin{equation}
  g_F^{\mathrm{quot}}
  =
  \begin{pmatrix}
    I(c) & -cV(d)I(c)\\
    -cV(d)I(c) & V(d)^2M_2(c)
  \end{pmatrix}.
  \label{eq:quotient-fisher}
\end{equation}

\begin{proposition}[Non-degeneracy of the identifiable quotient direction]
\label{prop:quotient-fisher-positive}
For $d>0$ and $c>0$, the Fisher metric
\eqref{eq:quotient-fisher} is positive definite.
Equivalently, after elimination of the nuisance coordinate $\beta$,
\begin{equation}
  I_{\mu\mu}^{\mathrm{eff}}
  =
  V(d)^2
  \bigl[
    M_2(c)-c^2I(c)
  \bigr]
  >0.
  \label{eq:mu-efficient-fisher}
\end{equation}
More explicitly,
\begin{equation}
  M_2(c)-c^2I(c)
  =
  \frac{
    4c^4\rho^4
    (\rho^2+3)(3\rho^2+1)
  }{
    (1-\rho^2)^4
  }
  >0.
  \label{eq:quotient-positive-factor}
\end{equation}
\end{proposition}

\begin{proof}
For $c>0$ one has $0<\rho<1$ and hence $I(c)>0$.
Moreover,
\(
d-1+e^{-d}>0
\)
for $d>0$, so $V(d)>0$.
The Schur complement of the $\beta$ block in
\eqref{eq:quotient-fisher} is exactly
\eqref{eq:mu-efficient-fisher}, and the factorization
\eqref{eq:quotient-positive-factor} is strictly positive.
\end{proof}

Thus the strict limiting experiment has two simultaneous features:
\begin{equation}
  S_\lambda=0,
  \qquad
  I_{\mu\mu}^{\mathrm{eff}}>0.
  \label{eq:fold-summary}
\end{equation}
The signed parametrization loses first-order rank, but the invariant quotient
coordinate remains statistically regular.
This distinction will become the geometric starting point of
Section~\ref{sec:emergent-quotient}.

\subsection{Finite texture as a weak unfolding}
\label{subsec:benchmark-unfolding-clue}

The exact symmetry holds only at $\varepsilon=0$.
For large but finite $\alpha$, the observation amplitude weakly reveals the
texture fluctuation and restores sensitivity to the sign of $\lambda$.
The first two efficient directions can be computed explicitly.

Let $\mathfrak r$ be defined by the leading observed score at the symmetric
point,
\begin{equation}
  \left.
  S_\lambda^Y(\varepsilon,\lambda)
  \right|_{\lambda=0}
  =
  \varepsilon u(d)\,\mathfrak r
  +
  o_{L^2}(\varepsilon),
  \qquad
  u(d)=\frac{1-e^{-d}}{d},
  \label{eq:r-score-definition}
\end{equation}
and retain $k$ from \eqref{eq:qhk-def}.
After projection on the orthogonal complement of the nuisance score $q$, set
\begin{equation}
  \mathfrak r^\perp=\mathfrak r-q,
  \qquad
  k^\perp=k+cq.
  \label{eq:efficient-rk}
\end{equation}
The projection coefficients follow from
\(
  \mathbb E[q\mathfrak r]=I(c)
\)
and \eqref{eq:qk-cross}.
The resulting efficient score has the local expansion
\begin{equation}
  S_{\lambda,\mathrm{eff}}
  =
  \varepsilon u(d)\,\mathfrak r^\perp
  +
  \lambda V(d)\,k^\perp
  +
  o_{L^2}(\varepsilon+|\lambda|).
  \label{eq:benchmark-score-unfolding}
\end{equation}
The explicit score functions are given in
Appendix~\ref{app:cir-ou-formulas}; for the present section, their Gram
geometry is the relevant fact.

\begin{proposition}[Independence of the odd and even efficient directions]
\label{prop:benchmark-gram-positive}
For every $c>0$,
\begin{equation}
  \det
  \operatorname{Gram}
  \bigl(
    \mathfrak r^\perp,
    k^\perp
  \bigr)
  =
  \frac{
    32c^6\rho^6
    (3\rho^4+2\rho^2+3)
  }{
    (1-\rho^2)^6
  }
  >0.
  \label{eq:benchmark-gram}
\end{equation}
Hence the sign-restoring and quotient-identifiable directions remain linearly
independent after nuisance elimination.
\end{proposition}

\begin{proof}
The closed-form moments of
\(
(\mathfrak r^\perp,k^\perp)
\)
give \eqref{eq:benchmark-gram}.
Since $c>0$ implies $0<\rho<1$, every factor on the right-hand side is
strictly positive.
\end{proof}

As a first consequence,
\begin{equation}
  I_{\lambda\lambda}^{\mathrm{eff}}(\varepsilon,0)
  =
  \varepsilon^2u(d)^2
  \mathbb E[(\mathfrak r^\perp)^2]
  +
  o(\varepsilon^2),
  \label{eq:finite-alpha-efficient-fisher}
\end{equation}
so the signed coupling direction is regular for all sufficiently small
$\varepsilon>0$ (equivalently, sufficiently large finite $\alpha$), while its
Fisher information collapses at order
\(
\varepsilon^2=\alpha^{-1}
\)
as the symmetric limit is approached.

Equation \eqref{eq:benchmark-score-unfolding} displays the two local scales
that will organize the rest of the paper.
The odd information is restored at order $\varepsilon$, whereas the even
quotient direction enters at order $\lambda$ in the score, or equivalently at
order $\lambda^2$ in the local statistical displacement.
In the transition layer
\(
  \lambda=\varepsilon\zeta
\),
both mechanisms therefore contribute at the same order.
Section~\ref{sec:emergent-quotient} first isolates the limiting reflection
quotient and its invariant tangent; Section~\ref{sec:weak-unfolding} then
shows how the odd and even directions combine in a general weak unfolding.
The corresponding statistical resolution is developed afterwards.

\section{Emergent reflection symmetry and the identifiable quotient}
\label{sec:emergent-quotient}

Section~\ref{sec:benchmark} exhibited the central phenomenon in the CIR--OU
benchmark: a signed parameter that is locally regular at finite texture becomes
identified with its negative in the strict Gaussian-texture limit, while the
invariant coordinate $\mu=\lambda^2/2$ retains non-degenerate Fisher
information.
We now abstract this mechanism from the benchmark.
The first step is measure-theoretic and requires no differentiability; the
second concerns the geometry of the resulting orbit space near the fixed
locus.

\subsection{A hidden-involution principle}
\label{subsec:hidden-involution}

Let $(\mathsf Z,\mathcal Z)$ be a latent measurable space and
$(\mathsf Y,\mathcal Y)$ an observation space.
Consider a family of observed laws
\begin{equation}
  P_{\varepsilon,\lambda}(A)
  =
  \int_{\mathsf Z}
  K_\lambda(z,A)\,\nu_\varepsilon(dz),
  \qquad
  A\in\mathcal Y,
  \label{eq:hidden-model}
\end{equation}
where $K_\lambda$ is a Markov kernel and $\nu_\varepsilon$ is the latent law.
Additional regular parameters and nuisances are suppressed temporarily.

Assume that the latent space carries a measurable involution
\begin{equation}
  \iota^2=\mathrm{id}_{\mathsf Z},
  \label{eq:iota-involution}
\end{equation}
and that the observation mechanism is equivariant in the sense that
\begin{equation}
  K_{-\lambda}(z,A)
  =
  K_\lambda(\iota z,A)
  \qquad
  \text{for all }A\in\mathcal Y.
  \label{eq:kernel-equivariance}
\end{equation}

\begin{proposition}[Hidden-involution identity]
\label{prop:hidden-involution}
Under \eqref{eq:iota-involution}--\eqref{eq:kernel-equivariance},
\begin{equation}
  P_{\varepsilon,-\lambda}(A)
  =
  \int_{\mathsf Z}
  K_\lambda(z,A)\,
  (\iota_\#\nu_\varepsilon)(dz).
  \label{eq:hidden-involution-identity}
\end{equation}
Consequently, if the limiting latent law is invariant,
\begin{equation}
  \iota_\#\nu_0=\nu_0,
  \label{eq:latent-invariance}
\end{equation}
then the limiting observed experiment satisfies
\begin{equation}
  P_{0,\lambda}=P_{0,-\lambda}.
  \label{eq:abstract-reflection}
\end{equation}
\end{proposition}

\begin{proof}
Using \eqref{eq:kernel-equivariance} and the definition of push-forward,
\[
  P_{\varepsilon,-\lambda}(A)
  =
  \int K_\lambda(\iota z,A)\,\nu_\varepsilon(dz)
  =
  \int K_\lambda(z,A)\,(\iota_\#\nu_\varepsilon)(dz).
\]
Equation \eqref{eq:abstract-reflection} follows immediately from
\eqref{eq:latent-invariance}.
\end{proof}

Proposition~\ref{prop:hidden-involution} is deliberately elementary.
Its role is to locate the origin of the observed symmetry.
The reflection need not be a symmetry of the full latent model for
$\varepsilon>0$.  It becomes an exact symmetry of the \emph{marginal observed
law} when the limiting latent distribution becomes invariant under $\iota$.
In the CIR--OU benchmark, $\iota$ is the sign reversal of the limiting Gaussian
texture fluctuation.
Section~\ref{sec:weak-unfolding} will formulate the symmetry-breaking jet
directly at the observed Hellinger level; departure of
$\iota_\#\nu_\varepsilon$ from $\nu_\varepsilon$ then provides one
sufficient latent mechanism for producing it.

\subsection{The reflection relation and the fixed locus}
\label{subsec:reflection-relation}

We now restore a regular parameter $\beta$ and work locally on
\begin{equation}
  M
  =
  U_\beta\times(-\delta,\delta)_\lambda,
  \label{eq:parameter-M}
\end{equation}
where $U_\beta$ is open in $\mathbb R^m$.
The limiting reflection acts by
\begin{equation}
  s(\beta,\lambda)
  =
  (\beta,-\lambda).
  \label{eq:parameter-reflection}
\end{equation}
Its fixed locus is
\begin{equation}
  F
  =
  U_\beta\times\{0\}.
  \label{eq:fixed-locus}
\end{equation}

For the remainder of this section we assume that, in the neighbourhood under
consideration, reflection is the only limiting non-identifiability:
\begin{equation}
  P_{0,\beta,\lambda}
  =
  P_{0,\tilde\beta,\tilde\lambda}
  \quad\Longleftrightarrow\quad
  \tilde\beta=\beta,
  \qquad
  \tilde\lambda=\pm\lambda.
  \label{eq:local-complete-identification}
\end{equation}
This local completeness condition is satisfied by the benchmark after
shrinking the neighbourhood of the symmetric point.
Indeed, Proposition~\ref{prop:quotient-fisher-positive} gives a
non-degenerate Fisher metric in the invariant coordinates $(\beta,\mu)$.
The corresponding differential of the quotient statistical map is therefore
injective at the base point; the constant-rank theorem then gives local
injectivity of the quotient parametrization after restricting to a sufficiently
small neighbourhood.
Exact reflection symmetry accounts for the two signed representatives
$\lambda$ and $-\lambda$.

The orbit equivalence relation is
\begin{equation}
  R
  =
  \Delta_M\cup\Gamma_s
  \subset M\times M,
  \label{eq:orbit-relation}
\end{equation}
where $\Delta_M$ is the diagonal and $\Gamma_s$ the graph of the reflection.
Away from $F$, the two branches are disjoint and the finite
$\mathbb Z_2$-action is free and proper.
At a fixed point, however, they meet.

A standard form of Godement's criterion states that an equivalence relation
defines a smooth quotient for which the canonical projection is a submersion
only if its graph is an embedded submanifold of $M\times M$ and the projection
of that graph onto $M$ is a submersion
\cite{Serre1992}.
The reflection relation fails the first condition at $F$.

\begin{proposition}[Failure of ordinary quotient regularity at the fixed locus]
\label{prop:godement-failure}
The relation $R$ in \eqref{eq:orbit-relation} is not a $C^1$ embedded
submanifold of $M\times M$ at points lying over $F$.
Hence the reflection quotient is not a regular quotient in the sense of
Godement at the fixed locus.
\end{proposition}

\begin{proof}
Write a point of $M\times M$ as
$(\beta,\lambda;\beta',\lambda')$.
Locally,
\begin{equation}
  R
  =
  \left\{
    \beta'=\beta,\quad
    (\lambda'-\lambda)(\lambda'+\lambda)=0
  \right\}.
  \label{eq:relation-local-equation}
\end{equation}
In the two normal coordinates
\begin{equation}
  u=\lambda'-\lambda,
  \qquad
  v=\lambda'+\lambda,
  \label{eq:uv-coordinates}
\end{equation}
the transverse slice of $R$ is therefore
\begin{equation}
  uv=0.
  \label{eq:normal-cross}
\end{equation}
It is the union of the two coordinate axes.
Equivalently, the tangent spaces of the two smooth branches are
\begin{align}
  T\Delta_M
  &=
  \{
    (\delta\beta,\delta\lambda;
     \delta\beta,\delta\lambda)
  \},
  \\
  T\Gamma_s
  &=
  \{
    (\delta\beta,\delta\lambda;
     \delta\beta,-\delta\lambda)
  \}.
\end{align}
At $\lambda=0$ their union has a tangent cone with two distinct normal
directions and is not a vector space.
Thus $R$ is not a $C^1$ embedded submanifold there.
\end{proof}

The obstruction is therefore attached to the fixed-point isotropy, not to the
behaviour away from it.
For $\lambda\neq0$, the quotient is locally an ordinary smooth quotient and
the orbit projection is a two-sheeted covering.

\subsection{Coarse quotient, reflection orbifold and invariant coordinate}
\label{subsec:coarse-quotient}

Although the ordinary submersion quotient fails at $F$, the orbit space has a
simple coarse description.
Define
\begin{equation}
  q:M\longrightarrow Q,
  \qquad
  q(\beta,\lambda)
  =
  \left(
    \beta,\,
    \mu=\frac{\lambda^2}{2}
  \right),
  \label{eq:coarse-quotient-map}
\end{equation}
with
\begin{equation}
  Q
  =
  U_\beta\times
  \left[0,\frac{\delta^2}{2}\right).
  \label{eq:Q-half-space}
\end{equation}
The map $q$ is constant precisely on the $\mathbb Z_2$-orbits and induces a
homeomorphism
\begin{equation}
  M/\mathbb Z_2
  \simeq Q.
  \label{eq:coarse-homeomorphism}
\end{equation}
Thus the coarse identifiable space is naturally represented as a smooth
manifold with boundary.

This representation must not be confused with regularity of the original
orbit projection.
Indeed,
\begin{equation}
  dq_{(\beta,0)}
  (\delta\beta,\delta\lambda)
  =
  (\delta\beta,0),
  \label{eq:dq-rank-drop}
\end{equation}
so $q$ loses one rank along $F$ and is not a submersion there.

If isotropy is retained rather than discarded, the finite reflection action
defines a proper \'etale action groupoid
\begin{equation}
  \mathbb Z_2\ltimes M
  \rightrightarrows M.
  \label{eq:action-groupoid}
\end{equation}
In orbifold conventions that admit reflections, this is the local model of a
reflection orbifold; equivalent conventions describe the fixed stratum as a
reflector boundary.
The groupoid remains smooth at the fixed locus even though its orbit relation,
viewed as the subset $R\subset M\times M$, has the crossing described in
Proposition~\ref{prop:godement-failure}.
For general background on finite-group quotient orbifolds and their groupoid
description, see \cite{AdemLeidaRuan2007}.

The coordinate $\mu$ is not an arbitrary desingularizing trick.
In the one-dimensional normal representation, $\mathbb Z_2$ acts by the sign
representation.
The invariant polynomial algebra is generated by $\lambda^2$, and Schwarz's
theorem on smooth invariants of compact group actions implies, in this case,
that every smooth reflection-invariant germ factors smoothly through
$\lambda^2$ \cite{Schwarz1975}.
With the normalization in \eqref{eq:coarse-quotient-map},
\begin{equation}
  f(\beta,\lambda)=f(\beta,-\lambda)
  \quad\Longrightarrow\quad
  f(\beta,\lambda)
  =
  \bar f
  \left(
    \beta,\frac{\lambda^2}{2}
  \right)
  \label{eq:smooth-invariant-factorization}
\end{equation}
for a smooth germ $\bar f$ on the half-space $Q$.

\subsection{Invariant jets and the first quotient tangent}
\label{subsec:invariant-jets}

The preceding factorization clarifies the order at which information can
survive the reflection.
At a fixed point $x_0=(\beta_0,0)$,
\begin{equation}
  T_{x_0}M
  =
  T_{x_0}F\oplus N,
  \label{eq:tangent-splitting}
\end{equation}
where the isotropy acts trivially on $T_{x_0}F$ and by sign on the
one-dimensional normal space $N$.
Hence
\begin{equation}
  \operatorname{Sym}^1(N^*)^{\mathbb Z_2}
  =
  0,
  \qquad
  \operatorname{Sym}^2(N^*)^{\mathbb Z_2}
  \simeq\mathbb R.
  \label{eq:normal-invariants}
\end{equation}
No nonzero invariant first normal jet is allowed, whereas a quadratic normal
jet is allowed.

To make this statistical, suppose the limiting family is dominated and let
\begin{equation}
  \Phi(\beta,\lambda)
  =
  2\sqrt{p_{0,\beta,\lambda}}
  \label{eq:hellinger-map-section3}
\end{equation}
be its Hellinger embedding.
Exact reflection symmetry implies
\begin{equation}
  \Phi(\beta,\lambda)
  =
  \bar\Phi
  \left(
    \beta,\frac{\lambda^2}{2}
  \right).
  \label{eq:hellinger-factorization}
\end{equation}
Consequently,
\begin{equation}
  \left.
  \partial_\lambda\Phi
  \right|_{\lambda=0}
  =
  0,
  \qquad
  \left.
  \partial_{\lambda\lambda}^2\Phi
  \right|_{\lambda=0}
  =
  \left.
  \partial_\mu\bar\Phi
  \right|_{\mu=0}.
  \label{eq:hellinger-jets}
\end{equation}

Equation \eqref{eq:hellinger-jets} also prevents a common misinterpretation.
Second-order identifiability in the signed coordinate refers to the ramified
map
\begin{equation}
  M\xrightarrow{q}Q.
  \label{eq:ramification-map}
\end{equation}
On the identifiable space $Q$ itself, $\mu$ is an ordinary first-order
coordinate.
Thus the relevant decomposition is
\begin{equation}
  M
  \xrightarrow{q}
  Q
  \xrightarrow{\bar\Phi}
  L^2,
  \label{eq:q-phi-factorization}
\end{equation}
and the order-two behaviour belongs to the first arrow, not to a failure of
first-order regularity of the second one.
In the benchmark,
Proposition~\ref{prop:quotient-fisher-positive} verifies explicitly that
$d\bar\Phi(\partial_\mu)$ remains nonzero after nuisance elimination.

This distinction is also useful in the presence of nuisance coordinates.
A smooth redefinition of the regular coordinate of the form
\begin{equation}
  \widetilde\beta
  =
  \beta
  +
  O(\lambda^2)
  \label{eq:nuisance-coordinate-mixing}
\end{equation}
can change a representative of the second normal jet by a tangent vector in
the nuisance direction.
The corresponding class modulo the nuisance tangent space is unchanged.
The efficient even direction represented by $k^\perp$ in the benchmark
of Section~\ref{sec:benchmark} is precisely a representative of this
quotient-identifiable class.

\subsection{Parametrization quotient versus Fisher geometry}
\label{subsec:orbifold-vs-fisher}

There is a final distinction that will matter later.
The geometry induced by statistical distinguishability on $Q$ is not the
same object as an ordinary Riemannian metric on the reflection orbifold.

To see this, let a non-degenerate metric on the identifiable coordinates be
written as
\begin{equation}
  g_Q
  =
  g_{\beta\beta}\,d\beta^2
  +
  2g_{\beta\mu}\,d\beta\,d\mu
  +
  g_{\mu\mu}\,d\mu^2.
  \label{eq:quotient-general-metric}
\end{equation}
Since
\(
  d\mu=\lambda\,d\lambda
\),
its pull-back to the signed parametrization is
\begin{equation}
  q^*g_Q
  =
  g_{\beta\beta}\,d\beta^2
  +
  2\lambda g_{\beta\mu}\,d\beta\,d\lambda
  +
  \lambda^2g_{\mu\mu}\,d\lambda^2.
  \label{eq:pullback-degenerate}
\end{equation}
Even when $g_Q$ is non-degenerate at $\mu=0$, its pull-back necessarily
degenerates in the normal $\lambda$ direction.
This is exactly the behaviour of the benchmark Fisher metric.

By contrast, start with an ordinary non-degenerate
$\mathbb Z_2$-invariant Riemannian metric upstairs containing the normal term
$d\lambda^2$.
On the free part $\mu>0$, that term is represented in coarse coordinates by
\begin{equation}
  d\lambda^2
  =
  \frac{d\mu^2}{2\mu},
  \label{eq:orbifold-metric-coarse}
\end{equation}
which is singular as a tensor in the boundary coordinate $\mu$.
The two behaviours are opposite.  An ordinary reflection-orbifold metric is
non-degenerate upstairs but singular in the coarse boundary chart, whereas the
quotient Fisher metric considered here is regular in the identifiable boundary
coordinate and degenerates after pull-back to the signed coordinate
$\lambda$.
\begin{equation}
  \text{orbifold regularity upstairs}
  \quad\longleftrightarrow\quad
  \text{Fisher regularity on the identifiable quotient}.
  \label{eq:orbifold-fisher-distinction}
\end{equation}
Accordingly, the Fisher metric of the identifiable experiment should not be
described as the ordinary orbifold metric inherited from a non-degenerate
metric on the signed parameter space.

We will use the following terminology.
The limiting statistical model is an
\emph{identified reflection quotient}, represented by the manifold with
boundary $Q$ in the invariant coordinate $\mu$.
The action groupoid records the reflection isotropy, while the Fisher metric
records statistical distinguishability on the coarse identifiable space.
Keeping these two structures separate is essential.

The quotient analysis is now complete on the exact symmetric face.
What it does not determine is how the missing signed direction reappears when
$\varepsilon>0$.
The next section addresses this transverse question directly at the observed
Hellinger level and then relates the resulting jets to a sufficient
hidden-latent mechanism.

\section{Weak symmetry breaking and the local normal form}
\label{sec:weak-unfolding}

The quotient analysis of Section~\ref{sec:emergent-quotient} concerns the exact
symmetric face $\varepsilon=0$.
We now ask how the missing signed direction reappears when that symmetry is
weakly broken.
We first derive the orders
\(
\varepsilon\lambda
\)
and
\(
\lambda^2
\)
directly from the observed statistical experiment.
We then show that a hidden latent involution provides one sufficient mechanism
for producing the corresponding jets.

\subsection{Observed Hellinger two-jet}
\label{subsec:hellinger-normal-form}

Let $\varpi$ dominate the observed laws in a neighbourhood of $(0,0)$ and
write
\begin{equation}
  \Phi(\varepsilon,\lambda)
  =
  2\sqrt{p_{\varepsilon,\lambda}}
  \in
  L^2(\varpi)
  \label{eq:observed-Hellinger-map}
\end{equation}
for the Hellinger embedding of the observed experiment.
We assume that $\Phi$ is twice continuously Fr\'echet differentiable at the
symmetric point as an $L^2(\varpi)$-valued map.
For a physical one-sided control parameter $\varepsilon\geq0$, the same
statements use one-sided derivatives; a signed extension is needed only when
we invoke the standard cross-cap terminology.

Assume further that the limiting observed family has the exact reflection
symmetry
\begin{equation}
  \Phi(0,\lambda)
  =
  \Phi(0,-\lambda)
  \label{eq:observed-reflection-Hellinger}
\end{equation}
for $|\lambda|$ sufficiently small.
Hence
\begin{equation}
  \partial_\lambda\Phi(0,0)=0.
  \label{eq:lambda-Hellinger-zero}
\end{equation}
Define the two Hellinger jets
\begin{equation}
  \mathfrak h_-
  =
  \partial_{\varepsilon\lambda}^2
  \Phi(0,0),
  \qquad
  \mathfrak h_+
  =
  \partial_{\lambda\lambda}^2
  \Phi(0,0).
  \label{eq:Hellinger-two-jets}
\end{equation}

\begin{theorem}[Observed weak-symmetry-breaking normal form]
\label{thm:weak-breaking-normal-form}
Under the assumptions above,
\begin{equation}
\begin{split}
  \Phi(\varepsilon,\lambda)
  -
  \Phi(\varepsilon,0)
  ={}&
  \varepsilon\lambda\,\mathfrak h_-
  +
  \frac{\lambda^2}{2}\,\mathfrak h_+
  \\
  &+
  o_{L^2(\varpi)}
  \left(
    \varepsilon|\lambda|+\lambda^2
  \right)
\end{split}
  \label{eq:observed-Hellinger-normal-form}
\end{equation}
as $(\varepsilon,\lambda)\to(0,0)$.
Thus the first signed direction that can reappear away from the symmetric
face is mixed, whereas the first direction surviving on the symmetric face is
even in $\lambda$.
\end{theorem}

\begin{proof}
By the fundamental theorem of calculus in the Hilbert space
$L^2(\varpi)$,
\(\Phi(\varepsilon,\lambda)-\Phi(\varepsilon,0)
=\lambda\int_0^1\partial_\lambda\Phi(\varepsilon,t\lambda)\,dt\).
Continuous Fr\'echet differentiability of
$\partial_\lambda\Phi$ gives, uniformly for $t\in[0,1]$,
\[
\partial_\lambda
\Phi(\varepsilon,t\lambda)
=
\partial_\lambda\Phi(0,0)
+
\varepsilon\mathfrak h_-
+
t\lambda\mathfrak h_+
+
o_{L^2}
\left(
  |\varepsilon|+|\lambda|
\right).
\]
Equation~\eqref{eq:lambda-Hellinger-zero} removes the constant term.
Integration over $t$ yields
\eqref{eq:observed-Hellinger-normal-form}.
\end{proof}

Let
\(
P_0=P_{0,0}
\)
and assume
\(
p_0>0
\)
on the support under consideration.
Differentiating the identity
\(
\|\Phi(\varepsilon,\lambda)\|_{L^2(\varpi)}^2=4
\)
shows that both vectors in
\eqref{eq:Hellinger-two-jets}
are orthogonal to
\(
\sqrt{p_0}
\).
They therefore admit unique score representatives
\begin{equation}
  \mathfrak h_-
  =
  \sqrt{p_0}\,J_-,
  \qquad
  \mathfrak h_+
  =
  \sqrt{p_0}\,J_+,
  \qquad
  J_-,J_+\in L_0^2(P_0).
  \label{eq:Jpm-def}
\end{equation}
The subscript records reflection parity:
$J_-$ is the mixed sign-restoring jet and $J_+$ the even quotient jet.
Equation~\eqref{eq:observed-Hellinger-normal-form} becomes
\begin{equation}
\begin{split}
  2\sqrt{p_{\varepsilon,\lambda}}
  -
  2\sqrt{p_{\varepsilon,0}}
  ={}&
  \sqrt{p_0}
  \left(
    \varepsilon\lambda J_-
    +
    \frac{\lambda^2}{2}J_+
  \right)
  \\
  &+
  o_{L^2(\varpi)}
  \left(
    \varepsilon|\lambda|+\lambda^2
  \right).
\end{split}
  \label{eq:hellinger-normal-form}
\end{equation}

On the symmetric face,
\begin{equation}
  2\sqrt{p_{0,\lambda}}
  -
  2\sqrt{p_{0,0}}
  =
  \frac{\lambda^2}{2}
  \sqrt{p_0}\,J_+
  +
  o_{L^2(\varpi)}(\lambda^2).
  \label{eq:hellinger-symmetric-face}
\end{equation}
Since
\(
\mu=\lambda^2/2
\),
the quotient score is $J_+$.
Thus the abstract even jet is exactly the first invariant tangent identified
in Section~\ref{subsec:invariant-jets}.

The normalized square-root likelihood-ratio form
\begin{equation}
  \sqrt{
    \frac{p_{\varepsilon,\lambda}}
         {p_{\varepsilon,0}}
  }
  =
  1
  +
  \frac12
  \left(
    \varepsilon\lambda J_-
    +
    \frac{\lambda^2}{2}J_+
  \right)
  +
  o_{L^2(P_0)}
  \left(
    \varepsilon|\lambda|+\lambda^2
  \right)
  \label{eq:normalized-Hellinger-normal-form}
\end{equation}
requires an additional local comparison between
$p_{\varepsilon,0}$ and $p_0$.
A convenient sufficient condition is given in
Appendix~\ref{app:hellinger-normal-form}.
The triangular-array theorem of Section~\ref{sec:local-asymptotics} assumes
the normalized expansion directly, so the geometric normal form itself does
not depend on that stronger comparison.

\subsection{A hidden-latent sufficient mechanism}
\label{subsec:anti-invariant-perturbation}

The observed theorem does not require a latent representation, but the hidden
involution of Section~\ref{subsec:hidden-involution} explains how the two jets
can arise.  Suppose
\begin{equation}
  P_{\varepsilon,\lambda}(dy)
  =
  \int K_\lambda(z,dy)\,\nu_\varepsilon(dz),
  \qquad
  K_{-\lambda}(z,\cdot)=K_\lambda(\iota z,\cdot),
  \label{eq:weak-latent-model}
\end{equation}
and, as a sufficient condition, that
\begin{equation}
  \frac{d\nu_\varepsilon}{d\nu_0}
  =1+\varepsilon a+o_{L^2}(\varepsilon),
  \qquad
  a\circ\iota=-a.
  \label{eq:latent-law-expansion}
\end{equation}
If the kernel admits a second-order operator expansion
$\mathcal K_\lambda=\mathcal K_0+\lambda\mathcal K_1+
\lambda^2\mathcal K_2/2+o(\lambda^2)$ in the weighted density topology of
Appendix~\ref{app:hellinger-normal-form}, equivariance gives
$\mathcal K_1 1=0$, $\mathcal K_0a=0$ and $\mathcal K_2a=0$.  Hence
\begin{equation}
  p_{\varepsilon,\lambda}-p_{\varepsilon,0}
  =
  \varepsilon\lambda\,\mathcal K_1a
  +\frac{\lambda^2}{2}\,\mathcal K_2 1
  +o(\varepsilon|\lambda|+\lambda^2),
  \label{eq:density-normal-form}
\end{equation}
and therefore
\begin{equation}
  J_- = \frac{\mathcal K_1a}{p_0},
  \qquad
  J_+ = \frac{\mathcal K_2 1}{p_0}.
  \label{eq:latent-Jpm-identification}
\end{equation}
Supplementary Appendix~S1 gives the quantitative operator assumptions and
remainder estimate.  The mechanism is sufficient rather than necessary; the
CIR--OU benchmark, in particular, is certified directly at the observed
likelihood level.

\subsection{Nuisance projection and efficient jets}
\label{subsec:efficient-jets}

Let
\begin{equation}
  \mathcal H
  =
  L^2_0(P_0)
  \label{eq:score-Hilbert}
\end{equation}
be the centered score Hilbert space.
Suppose that a regular nuisance parameter
\(
\eta\in\mathbb R^r
\)
is present and let
\(
\mathcal N\subset\mathcal H
\)
be the closed span of its score directions at the symmetric point.
Denote the orthogonal projection onto $\mathcal N$ by
\(
\Pi_{\mathcal N}
\).
The efficient odd and even jets are
\begin{equation}
  J_-^\perp
  =
  (I-\Pi_{\mathcal N})J_-,
  \qquad
  J_+^\perp
  =
  (I-\Pi_{\mathcal N})J_+.
  \label{eq:efficient-Jpm}
\end{equation}

Projecting the score representatives in
\eqref{eq:hellinger-normal-form} gives the effective local displacement
\begin{equation}
  \varepsilon\lambda J_-^\perp
  +
  \frac{\lambda^2}{2}J_+^\perp.
  \label{eq:efficient-displacement}
\end{equation}
In particular,
\begin{equation}
  S_{\lambda,\mathrm{eff}}(\varepsilon,0)
  =
  \varepsilon J_-^\perp
  +
  o_{L^2}(\varepsilon),
  \label{eq:efficient-lambda-score}
\end{equation}
and therefore
\begin{equation}
  I_{\lambda\lambda}^{\mathrm{eff}}
  (\varepsilon,0)
  =
  \varepsilon^2
  \|J_-^\perp\|^2
  +
  o(\varepsilon^2).
  \label{eq:efficient-lambda-information}
\end{equation}
On the symmetric quotient,
\begin{equation}
  S_{\mu,\mathrm{eff}}
  =
  J_+^\perp,
  \qquad
  I_{\mu\mu}^{\mathrm{eff}}
  =
  \|J_+^\perp\|^2.
  \label{eq:efficient-mu-information}
\end{equation}

The key non-degeneracy condition is
\begin{equation}
  \Delta_{\mathrm{eff}}
  :=
  \|J_-^\perp\|^2\|J_+^\perp\|^2
  -
  \langle J_-^\perp,J_+^\perp\rangle^2
  >0.
  \label{eq:effective-gram}
\end{equation}
It says more than the separate non-vanishing of the two jets.
After nuisance elimination, the sign-restoring tangent cannot be absorbed
into the quotient-identifiable tangent, and conversely.
The weak unfolding is therefore genuinely two-directional.

This formulation is invariant under changes of nuisance coordinates.
Indeed, modifying a representative of the mixed or even jet by a nuisance
score changes neither its class in
\(
\mathcal H/\mathcal N
\)
nor the Gram determinant
\eqref{eq:effective-gram}.
The effective data are the two classes
\begin{equation}
  [J_-],[J_+]
  \in
  \mathcal H/\mathcal N,
  \label{eq:Jpm-quotient-classes}
\end{equation}
for which $J_-^\perp$ and $J_+^\perp$ are the orthogonal representatives.

\subsection{The effective cross-cap boundary jet}
\label{subsec:cross-cap}

The two efficient jets admit a simple geometric organization.
Let
\(
\mathcal I_H:\mathcal H\to L^2(\varpi)
\)
be the Hellinger tangent isometry
\begin{equation}
  \mathcal I_H f=\sqrt{p_0}\,f,
  \qquad
  \|\mathcal I_H f\|_{L^2(\varpi)}
  =
  \|f\|_{L^2(P_0)},
  \label{eq:Hellinger-tangent-isometry}
\end{equation}
and let
\begin{equation}
  \mathcal E_H
  =
  \mathcal I_H(\mathcal H)/
  \mathcal I_H(\mathcal N)
  \label{eq:efficient-Hellinger-space}
\end{equation}
be the corresponding efficient Hellinger tangent space.
Consider the map
\begin{equation}
  \Phi^{\mathrm{eff}}(\varepsilon,\lambda)
  =
  \left(
    \varepsilon,\,
    \left[
      2\sqrt{p_{\varepsilon,\lambda}}
      -
      2\sqrt{p_{\varepsilon,0}}
    \right]
  \right)
  \in
  \mathbb R\oplus\mathcal E_H.
  \label{eq:effective-Hellinger-map}
\end{equation}
Here brackets denote the class modulo the nuisance Hellinger tangent space.
Its second jet at the symmetric point is, under the identification
\(\mathcal I_H:\mathcal H/\mathcal N\simeq\mathcal E_H\),
\begin{equation}
  j^2\Phi^{\mathrm{eff}}
  (\varepsilon,\lambda)
  =
  \left(
    \varepsilon,\,
    \varepsilon\lambda J_-^\perp
    +
    \frac{\lambda^2}{2}J_+^\perp
  \right).
  \label{eq:effective-Hellinger-2jet}
\end{equation}

\begin{proposition}[Cross-cap form of the non-degenerate unfolding]
\label{prop:cross-cap}
Under the assumptions of
Theorem~\ref{thm:weak-breaking-normal-form}, suppose in addition that
\(
\Delta_{\mathrm{eff}}>0
\).
Then, after an invertible linear change of coordinates in the two-dimensional
target plane
\(
\operatorname{span}\{J_-^\perp,J_+^\perp\}
\),
the essential second jet of
\(
\Phi^{\mathrm{eff}}
\)
is
\begin{equation}
  (\varepsilon,\lambda)
  \longmapsto
  \left(
    \varepsilon,\,
    \varepsilon\lambda,\,
    \frac{\lambda^2}{2}
  \right).
  \label{eq:cross-cap-normal-form}
\end{equation}
For a signed unfolding parameter this is the standard cross-cap $2$-jet up to
smooth linear rescalings.
If only $\varepsilon\geq0$ is physically available, the statistical family
realizes its one-sided source restriction.
\end{proposition}

\begin{proof}
Condition \eqref{eq:effective-gram} implies that
\(
J_-^\perp
\)
and
\(
J_+^\perp
\)
form a basis of their two-dimensional span.
Choose target coordinates sending these two vectors to the coordinate basis.
Equation \eqref{eq:effective-Hellinger-2jet} then becomes
\eqref{eq:cross-cap-normal-form}, up to nonzero linear rescalings of the two
target coordinates.
\end{proof}

No claim of full $\mathcal A$-equivalence of the infinite-dimensional
statistical map is needed here.
The statement concerns only the essential $2$-jet after projection onto the
two statistically active directions.
Its role is to organize the local geometry and, in particular, to expose the
distinguished layer
\begin{equation}
  |\lambda|\asymp\varepsilon,
  \label{eq:cross-cap-layer}
\end{equation}
in which the mixed and even displacements have the same order.

For fixed $\varepsilon>0$, the slice of
\eqref{eq:cross-cap-normal-form} is a regular parabola whose tangent at
$\lambda=0$ is the symmetry-breaking direction.
At $\varepsilon=0$, the first coordinate in the statistical plane collapses
and only the quotient half-line remains.
The cross-cap therefore provides a geometric interpolation between a regular
signed family and its limiting reflection quotient.

\subsection{Exact certification in the CIR--OU benchmark}
\label{subsec:benchmark-unfolding-certification}

For the benchmark of Section~\ref{sec:benchmark}, the observed Hellinger
hypotheses are verified directly from the likelihood.
The stronger latent Radon--Nikodym criterion of Supplementary Appendix~S1 is
therefore not needed.
The fixed-horizon DQM calculation and its one-transition specialization in
Appendix~\ref{app:cir-ou-formulas} identify
\begin{equation}
  J_-^\perp
  =
  u(d)\,\mathfrak r^\perp,
  \qquad
  J_+^\perp
  =
  V(d)\,k^\perp,
  \label{eq:benchmark-Jpm-identification}
\end{equation}
where
\begin{equation}
  u(d)=\frac{1-e^{-d}}{d},
  \qquad
  V(d)=\frac{2(d-1+e^{-d})}{d^2}.
  \label{eq:uV-recall}
\end{equation}
The benchmark Gram determinant of
\(
(\mathfrak r^\perp,k^\perp)
\)
is strictly positive for every $c>0$ by
Proposition~\ref{prop:benchmark-gram-positive}.
Since $u(d)>0$ and $V(d)>0$ for $d>0$,
\begin{equation}
  \Delta_{\mathrm{eff}}
  =
  u(d)^2V(d)^2
  \det
  \operatorname{Gram}
  \bigl(
    \mathfrak r^\perp,k^\perp
  \bigr)
  >0
  \qquad
  (d,c>0).
  \label{eq:benchmark-effective-gram}
\end{equation}
Thus the solvable model satisfies the non-degeneracy hypothesis of
Proposition~\ref{prop:cross-cap} throughout the fixed positive
$(d,c)$ domain considered here.

In the transition layer
\(
\lambda=\varepsilon\zeta
\),
the effective displacement becomes
\begin{equation}
  \varepsilon^2
  \left[
    \zeta J_-^\perp
    +
    \frac{\zeta^2}{2}J_+^\perp
  \right]
  +o_{L^2}(\varepsilon^2).
  \label{eq:critical-layer-displacement}
\end{equation}
This identifies the critical geometric scale, of order $\varepsilon^2$.
Section~\ref{sec:local-asymptotics} now compares that scale with the ordinary
$n^{-1/2}$ statistical resolution.

\section{Local asymptotic resolution of the unfolding}
\label{sec:local-asymptotics}

Section~\ref{sec:weak-unfolding} identified the effective local displacement
\(\varepsilon\lambda J_-^\perp+\lambda^2J_+^\perp/2\).
For a sample of size $n$, its resolvability is governed by a single
two-dimensional resolution map, which yields three distinct asymptotic regimes.

\subsection{The statistical resolution map}
\label{subsec:resolution-map}

For independent observations, the natural local scale of a regular Hellinger
tangent is $n^{-1/2}$.
The weak-unfolding normal form therefore suggests comparing
\(\varepsilon_n\lambda_n\) and \(\lambda_n^2/2\) after multiplication by
$\sqrt n$.
Define
\begin{equation}
  v(\varepsilon,\lambda)
  =
  \begin{pmatrix}
    \varepsilon\lambda\\[1mm]
    \lambda^2/2
  \end{pmatrix},
  \qquad
  u_n
  =
  \sqrt n\,
  v(\varepsilon_n,\lambda_n).
  \label{eq:resolution-vector}
\end{equation}
Thus the fundamental statistical resolution map is
\begin{equation}
  (\varepsilon_n,\lambda_n)
  \longmapsto
  \sqrt n
  \begin{pmatrix}
    \varepsilon_n\lambda_n\\[1mm]
    \lambda_n^2/2
  \end{pmatrix}.
  \label{eq:statistical-resolution-map}
\end{equation}

The relative visibility of the two components is controlled by
\begin{equation}
  \tau_n
  =
  \sqrt n\,\varepsilon_n^2.
  \label{eq:tau-def}
\end{equation}
Indeed, in the geometric transition layer
\(
\lambda_n=\varepsilon_n\zeta
\),
equation \eqref{eq:resolution-vector} becomes
\begin{equation}
  u_n
  =
  \tau_n
  \begin{pmatrix}
    \zeta\\[1mm]
    \zeta^2/2
  \end{pmatrix}.
  \label{eq:critical-resolution-vector}
\end{equation}
The quantity $\tau_n$ therefore compares the order-$\varepsilon_n^2$
cross-cap displacement with the ordinary statistical resolution
$n^{-1/2}$.

\subsection{A triangular-array local asymptotic theorem}
\label{subsec:master-triangular-array}

We formulate the result after nuisance elimination.
Precisely, when regular nuisance parameters are present, the theorem is
applied first to the joint target--nuisance experiment and then to its
efficient Gaussian quotient; equivalently one follows the least-favourable
local nuisance path.
Appendix~\ref{app:local-asymptotics} gives this reduction explicitly.
For each $\varepsilon$, let
\(V_\varepsilon=(J_{-,\varepsilon}^\perp,J_{+,\varepsilon}^\perp)^\top\)
be a centered two-component efficient score under the baseline law
$P_{\varepsilon,0}$, and let
\begin{equation}
  G_\varepsilon
  =
  \mathbb E_{\varepsilon,0}
  \left[
    V_\varepsilon V_\varepsilon^\top
  \right].
  \label{eq:G-epsilon}
\end{equation}
Theorem~\ref{thm:weak-breaking-normal-form} identifies the two limiting
Hellinger directions at the symmetric point.
Convergence of finite-$\varepsilon$ efficient score representatives to these
directions is an additional local regularity requirement; it is verified
directly in the CIR--OU benchmark and is included in the assumptions below.

\begin{theorem}[Triangular-array LAN in the two jet coordinates]
\label{thm:two-jet-LAN}
Let $\varepsilon_n\downarrow0$ and $\lambda_n\to0$.
Assume that, uniformly along local sequences for which
\(
\sqrt n\,\|v(\varepsilon_n,\lambda_n)\|
\)
remains bounded,
\begin{equation}
  \sqrt{
    \frac{
      dP_{\varepsilon_n,\lambda_n}
    }{
      dP_{\varepsilon_n,0}
    }
  }
  =
  1
  +
  \frac12
  v(\varepsilon_n,\lambda_n)^\top
  V_{\varepsilon_n}
  +
  r_{n},
  \label{eq:triangular-hellinger}
\end{equation}
with
\begin{equation}
  n\,\mathbb E_{\varepsilon_n,0}[r_n^2]
  \longrightarrow0.
  \label{eq:triangular-remainder}
\end{equation}
Assume further that
\begin{equation}
  G_{\varepsilon_n}
  \longrightarrow
  G,
  \qquad
  G>0,
  \label{eq:G-limit}
\end{equation}
and that the triangular array
\(
V_{\varepsilon_n,1},\ldots,V_{\varepsilon_n,n}
\)
satisfies a Lindeberg condition; a uniform
$L^{2+\eta}$ bound for some $\eta>0$ is sufficient.

If
\begin{equation}
  u_n
  =
  \sqrt n\,
  v(\varepsilon_n,\lambda_n)
  \longrightarrow u\in\mathbb R^2,
  \label{eq:u-limit}
\end{equation}
then, under
\(
P_{\varepsilon_n,0}^{\otimes n}
\),
\begin{equation}
  \log
  \frac{
    dP_{\varepsilon_n,\lambda_n}^{\otimes n}
  }{
    dP_{\varepsilon_n,0}^{\otimes n}
  }
  \ \Longrightarrow\
  u^\top Z
  -
  \frac12u^\top Gu,
  \qquad
  Z\sim N(0,G).
  \label{eq:master-LAN-limit}
\end{equation}
Equivalently, the local limit experiment is the two-dimensional Gaussian
shift with information matrix $G$, restricted to the limiting image of the
resolution map \eqref{eq:statistical-resolution-map}.
\end{theorem}

\begin{proof}
Set
\begin{equation}
  \Delta_n
  =
  \frac1{\sqrt n}
  \sum_{i=1}^n
  V_{\varepsilon_n,i}.
  \label{eq:Delta-n}
\end{equation}
By \eqref{eq:G-limit} and the triangular-array Lindeberg condition,
\begin{equation}
  \Delta_n
  \Longrightarrow
  Z\sim N(0,G).
  \label{eq:Delta-CLT}
\end{equation}
The Hellinger expansion
\eqref{eq:triangular-hellinger} and
\eqref{eq:triangular-remainder} imply, by
Lemma~\ref{lem:appB-quadratic-LR}, the quadratic log-likelihood expansion
\begin{equation}
  \Lambda_n
  =
  u_n^\top\Delta_n
  -
  \frac12
  u_n^\top
  \left[
    \frac1n
    \sum_{i=1}^n
    V_{\varepsilon_n,i}
    V_{\varepsilon_n,i}^{\top}
  \right]
  u_n
  +
  o_{P_{\varepsilon_n,0}}(1).
  \label{eq:LAN-quadratic-expansion}
\end{equation}
The empirical Gram matrix converges in probability to $G$ under the same
moment assumptions.
Combining this convergence with
\eqref{eq:u-limit} and \eqref{eq:Delta-CLT} gives
\eqref{eq:master-LAN-limit}.
\end{proof}

The theorem isolates the key point:
the asymptotic experiment is regular in the two jet coordinates
\((\varepsilon\lambda,\lambda^2/2)\), while the original signed parameter $\lambda$ traces a nonlinear subset of
that regular Gaussian experiment.

\subsection{Three asymptotic regimes}
\label{subsec:three-regimes}

The three regimes follow directly from
Theorem~\ref{thm:two-jet-LAN}.

\paragraph{Regular regime: $\tau_n\to\infty$.}
Choose
\begin{equation}
  \lambda_n
  =
  \frac{t}{\sqrt n\,\varepsilon_n},
  \qquad
  t\in\mathbb R.
  \label{eq:regular-local-scale}
\end{equation}
Then
\begin{equation}
  u_n
  =
  \begin{pmatrix}
    t\\[1mm]
    t^2/(2\tau_n)
  \end{pmatrix}
  \longrightarrow
  \begin{pmatrix}
    t\\0
  \end{pmatrix}.
  \label{eq:regular-u-limit}
\end{equation}
Only the sign-restoring direction remains visible.
The local experiment is ordinary one-dimensional LAN in the signed parameter
$\lambda$, with information
\(
g_{11}=\|J_-^\perp\|^2
\).
The corresponding local scale is
\begin{equation}
  \lambda_n
  \asymp
  \frac{1}{\sqrt n\,\varepsilon_n}.
  \label{eq:regular-rate}
\end{equation}

\paragraph{Quotient regime: $\tau_n\to0$.}
Set
\begin{equation}
  \lambda_n
  =
  n^{-1/4}z,
  \qquad
  z\in\mathbb R.
  \label{eq:quotient-local-scale}
\end{equation}
Then
\begin{equation}
  u_n
  =
  \begin{pmatrix}
    n^{1/4}\varepsilon_n z\\[1mm]
    z^2/2
  \end{pmatrix}
  \longrightarrow
  \begin{pmatrix}
    0\\
    z^2/2
  \end{pmatrix}.
  \label{eq:quotient-u-limit}
\end{equation}
The limiting likelihood depends on $z$ only through $z^2$.
Writing
\begin{equation}
  v=\frac{z^2}{2}\geq0
  \label{eq:quotient-v}
\end{equation}
turns the limit into an ordinary one-sided Gaussian shift in the identifiable
coordinate $v$, or equivalently in the local quotient coordinate
\(
\sqrt n\,\mu_n
\).
Thus the $n^{-1/4}$ scale in $\lambda$ is the pull-back of the regular
$n^{-1/2}$ scale in
\(
\mu=\lambda^2/2
\).

\paragraph{Critical regime: $\tau_n\to\tau\in(0,\infty)$.}
Set
\begin{equation}
  \lambda_n
  =
  \varepsilon_n\zeta.
  \label{eq:critical-lambda}
\end{equation}
Then
\begin{equation}
  u_n
  \longrightarrow
  \tau
  \begin{pmatrix}
    \zeta\\[1mm]
    \zeta^2/2
  \end{pmatrix}.
  \label{eq:critical-u-limit}
\end{equation}
Let $Z(\cdot)$ denote the isonormal Gaussian process on
\(
\operatorname{span}\{J_-^\perp,J_+^\perp\}
\),
and define
\begin{equation}
  h_\zeta^\perp
  =
  \zeta J_-^\perp
  +
  \frac{\zeta^2}{2}J_+^\perp.
  \label{eq:h-zeta}
\end{equation}
The limiting log-likelihood process is
\begin{equation}
  \Lambda(\zeta)
  =
  \tau Z(h_\zeta^\perp)
  -
  \frac{\tau^2}{2}
  \|h_\zeta^\perp\|^2.
  \label{eq:critical-curved-LAN}
\end{equation}
Hence the critical experiment is a curved Gaussian subexperiment of the ambient
two-dimensional LAN limit, with mean constrained
to the parabola
\begin{equation}
  m_\tau(\zeta)
  =
  \tau
  \left(
    \zeta J_-^\perp
    +
    \frac{\zeta^2}{2}J_+^\perp
  \right).
  \label{eq:critical-mean-curve}
\end{equation}
If
\(
\Delta_{\mathrm{eff}}>0
\),
this curve is not contained in an affine line.
The critical limit is therefore a genuinely curved one-parameter Gaussian
subexperiment; scalar LAN in $\zeta$ is not being claimed.
Under the compact-uniform Sobolev remainder condition stated in
Appendix~\ref{app:iid-process-limit}, this convergence holds in
$C(K)$ for every compact interval $K\subset\mathbb R$.

The three regimes can be summarized as
\begin{equation}
  \boxed{
  \begin{array}{ccl}
    \tau_n\to\infty
    &:&
    \text{regular signed experiment},\\[1mm]
    \tau_n\to\tau\in(0,\infty)
    &:&
    \text{curved Gaussian crossover},\\[1mm]
    \tau_n\to0
    &:&
    \text{reflection quotient}.
  \end{array}}
  \label{eq:three-regime-summary}
\end{equation}

\subsection{Curvature of the critical Gaussian experiment}
\label{subsec:critical-curvature}

The Gram non-degeneracy has a quantitative interpretation in terms of
statistical curvature.
Let
\begin{equation}
  G
  =
  \begin{pmatrix}
    g_{11} & g_{12}\\
    g_{12} & g_{22}
  \end{pmatrix}
  =
  \operatorname{Gram}
  (J_-^\perp,J_+^\perp).
  \label{eq:critical-G}
\end{equation}
In a Gaussian shift with identity covariance, Efron's statistical curvature
of a one-dimensional mean curve agrees with its Euclidean curvature in the
Fisher metric \cite{Efron1975}.
For the critical curve \eqref{eq:critical-mean-curve},
\begin{equation}
  m_\tau'(0)=\tau J_-^\perp,
  \qquad
  m_\tau''(0)=\tau J_+^\perp.
  \label{eq:critical-derivatives}
\end{equation}
Therefore
\begin{equation}
  \gamma_{\mathrm{crit}}(\tau)
  =
  \frac{
    \sqrt{
      \|m_\tau'(0)\|^2
      \|m_\tau''(0)\|^2
      -
      \langle
        m_\tau'(0),m_\tau''(0)
      \rangle^2
    }
  }{
    \|m_\tau'(0)\|^3
  }
  =
  \frac1{\tau}
  \frac{
    \sqrt{\det G}
  }{
    g_{11}^{3/2}
  }.
  \label{eq:critical-Efron-curvature}
\end{equation}
Thus
\(
\Delta_{\mathrm{eff}}>0
\)
is equivalent to strictly positive critical statistical curvature.
The factor $1/\tau$ has a simple interpretation:
multiplying the entire Gaussian mean curve by $\tau$ increases its statistical
size while decreasing its geometric curvature by the inverse scale factor.

This curvature belongs to the \emph{critical Gaussian subexperiment}.
It must not be confused with the intrinsic Fisher curvature of the limiting
quotient studied later in
Section~\ref{sec:quotient-information-geometry}.
The two objects live on different statistical manifolds and depend on
different jets.

\subsection{Exact crossover in the CIR--OU benchmark}
\label{subsec:benchmark-crossover}

For the texture--coherence benchmark,
\begin{equation}
  \varepsilon_n
  =
  \alpha_n^{-1/2},
  \qquad
  \tau_n
  =
  \frac{\sqrt n}{\alpha_n}.
  \label{eq:benchmark-tau}
\end{equation}
Hence
\begin{equation}
  \boxed{
  \begin{array}{ccl}
    \alpha_n\ll\sqrt n
    &:&
    \text{regular signed regime},\\[1mm]
    \alpha_n\asymp\sqrt n
    &:&
    \text{parabolic Gaussian regime},\\[1mm]
    \alpha_n\gg\sqrt n
    &:&
    \mathbb Z_2\text{ quotient regime}.
  \end{array}}
  \label{eq:benchmark-regimes}
\end{equation}
In the last regime,
\(
\lambda_n\asymp n^{-1/4}
\),
whereas in the first one
\(
\lambda_n\asymp\sqrt{\alpha_n/n}
\).

For independent repetitions of the stationary pair experiment,
Section~\ref{sec:benchmark} and
Proposition~\ref{prop:benchmark-gram-positive} give
\begin{equation}
  J_-^\perp
  =
  u(d)\,\mathfrak r^\perp,
  \qquad
  J_+^\perp
  =
  V(d)\,k^\perp.
  \label{eq:benchmark-critical-jets}
\end{equation}
Writing
\begin{equation}
  R
  =
  \|\mathfrak r^\perp\|^2,
  \qquad
  K
  =
  \|k^\perp\|^2,
  \qquad
  C
  =
  \langle
    \mathfrak r^\perp,k^\perp
  \rangle,
  \label{eq:RKC}
\end{equation}
the unscaled curvature factor is
\begin{equation}
  \gamma_0(d,c)
  =
  \frac{
    V(d)\sqrt{RK-C^2}
  }{
    u(d)^2R^{3/2}
  }
  >0.
  \label{eq:gamma0-RKC}
\end{equation}
Using the exact Gaussian moments from
Appendix~\ref{app:cir-ou-formulas},
this simplifies to
\begin{equation}
  \gamma_0(d,c)
  =
  \frac{V(d)\sqrt{2}}{u(d)^2}
  \,
  \frac{
    \sqrt{3\rho^4+2\rho^2+3}
  }{
    (3\rho^2+2)^{3/2}
  },
  \qquad
  \rho=e^{-c}.
  \label{eq:gamma0-explicit}
\end{equation}
If
\begin{equation}
  \frac{\alpha_n}{\sqrt n}
  \longrightarrow
  \kappa\in(0,\infty),
  \label{eq:kappa-critical}
\end{equation}
then
\(
\tau=1/\kappa
\)
and the curvature of the limiting Gaussian experiment is
\begin{equation}
  \gamma_{\mathrm{crit}}
  =
  \kappa\,\gamma_0(d,c)
  >0.
  \label{eq:benchmark-critical-curvature}
\end{equation}

The independent-repetition problem is therefore completely described by the
same geometry already visible at the one-experiment level:
the cross-cap two-jet is magnified by the statistical resolution
$\sqrt n$.
Section~\ref{sec:predictive} tests whether this crossover survives along a
single correlated trajectory.

\section{Predictive experiments: persistence under memory}
\label{sec:predictive}

The independent-repetition analysis of
Section~\ref{sec:local-asymptotics} isolates the local geometry of the weak
unfolding, but the original texture--coherence model is a time series.
We now show that the same two-jet mechanism survives when the likelihood is
factorized predictively along a single dependent trajectory.
In this factorization, memory changes the inner product seen by the two jets,
but not their parity structure or the central $\sqrt n$ resolution.

\subsection{Predictive likelihood and martingale jets}
\label{subsec:predictive-factorization}

Let
\(
(Y_k)_{k\geq0}
\)
be a stationary observed process and
\(
\mathcal F^Y_k=\sigma(Y_j:j\leq k)
\)
its natural filtration.
For a parameter pair $(\varepsilon,\lambda)$, factorize the observed
likelihood into predictive densities,
\begin{equation}
  p_{\varepsilon,\lambda}(Y_{0:n})
  =
  p_{\varepsilon,\lambda}(Y_0)
  \prod_{k=1}^n
  p_{\varepsilon,\lambda}
  (Y_k\mid\mathcal F^Y_{k-1}).
  \label{eq:predictive-factorization}
\end{equation}
The initial term is asymptotically negligible under stationarity and will be
suppressed below.

We concentrate on the critical geometric layer
\begin{equation}
  \lambda=\varepsilon\zeta,
  \qquad
  \delta=\varepsilon^2,
  \label{eq:predictive-critical-layer}
\end{equation}
with $\zeta$ in a fixed compact set.
Define the predictive density ratio
\begin{equation}
  R_{k,\varepsilon,\zeta}
  =
  \frac{
    p_{\varepsilon,\varepsilon\zeta}
    (Y_k\mid\mathcal F^Y_{k-1})
  }{
    p_{\varepsilon,0}
    (Y_k\mid\mathcal F^Y_{k-1})
  }.
  \label{eq:predictive-ratio}
\end{equation}
The predictive analogue of the local normal form is the following
quadratic-mean density-ratio expansion:
\begin{equation}
  R_{k,\varepsilon,\zeta}
  =
  1
  +
  \delta H_{k,\zeta}
  +
  \delta r_{k,\varepsilon,\zeta},
  \qquad
  H_{k,\zeta}
  =
  \zeta J_{-,k}
  +
  \frac{\zeta^2}{2}J_{+,k}.
  \label{eq:predictive-DQM}
\end{equation}
The two coefficients
$J_{-,k}$ and $J_{+,k}$ are derivatives of normalized predictive
densities.
Differentiating the conditional normalization
\(
\int
p_{\varepsilon,\lambda}
(y\mid\mathcal F^Y_{k-1})\,dy
=
1
\)
at the symmetric point therefore gives
\begin{equation}
  \mathbb E[
    J_{-,k}\mid\mathcal F^Y_{k-1}
  ]
  =
  0,
  \qquad
  \mathbb E[
    J_{+,k}\mid\mathcal F^Y_{k-1}
  ]
  =
  0.
  \label{eq:AB-martingale}
\end{equation}
Hence
\begin{equation}
  \mathbb E[
    H_{k,\zeta}
    \mid
    \mathcal F^Y_{k-1}
  ]
  =
  0.
  \label{eq:H-martingale}
\end{equation}
Finally, normalization of
\eqref{eq:predictive-DQM}
then implies
\begin{equation}
  \mathbb E[
    r_{k,\varepsilon,\zeta}
    \mid
    \mathcal F^Y_{k-1}
  ]
  =
  0.
  \label{eq:predictive-remainder-centered}
\end{equation}
Thus both the leading jets and the normalized remainder inherit a
martingale-difference structure even though the observations themselves are
dependent.

The martingale structure is the key simplification.
Long-range dependence enters through the predictable construction of
$J_{-,k}$ and $J_{+,k}$, while cross-time covariances of the score increments
vanish.

\subsection{Predictive Gram geometry and nuisance elimination}
\label{subsec:predictive-gram}

Write
\begin{equation}
  V_k
  =
  \begin{pmatrix}
    J_{-,k}\\
    J_{+,k}
  \end{pmatrix}.
  \label{eq:predictive-Vk}
\end{equation}
The relevant quadratic object is the predictable Gram matrix
\begin{equation}
  G_n^{\mathrm{pred}}
  =
  \frac1n
  \sum_{k=1}^n
  \mathbb E
  \left[
    V_kV_k^\top
    \mid
    \mathcal F^Y_{k-1}
  \right].
  \label{eq:predictable-Gram}
\end{equation}
Under stationarity and ergodicity, the natural limiting assumption is
\begin{equation}
  G_n^{\mathrm{pred}}
  \xrightarrow{P}
  G_{\mathrm{pred}}.
  \label{eq:predictive-Gram-limit}
\end{equation}

If a nuisance score increment
\(
Q_k\in\mathbb R^r
\)
is present, let the joint limiting Gram matrix be partitioned as
\begin{equation}
  G_{\mathrm{joint}}
  =
  \begin{pmatrix}
    G_{VV} & G_{VQ}\\
    G_{QV} & G_{QQ}
  \end{pmatrix}.
  \label{eq:predictive-joint-Gram}
\end{equation}
Whenever $G_{QQ}$ is invertible, the efficient predictive Gram matrix is the
Schur complement
\begin{equation}
  G_{\mathrm{pred}}^{\mathrm{eff}}
  =
  G_{VV}
  -
  G_{VQ}
  G_{QQ}^{-1}
  G_{QV}.
  \label{eq:predictive-efficient-Gram}
\end{equation}
Equivalently, one may work from the outset with efficient predictive
martingale differences
\(
J_{-,k}^\perp,J_{+,k}^\perp
\)
whose limiting Gram matrix is
\eqref{eq:predictive-efficient-Gram}.

The condition
\begin{equation}
  \det
  G_{\mathrm{pred}}^{\mathrm{eff}}
  >0
  \label{eq:predictive-nondegeneracy}
\end{equation}
has the same interpretation as in the independent experiment:
after conditioning on the entire observed past and eliminating nuisance
directions, the odd symmetry-breaking jet and the even quotient jet remain
genuinely independent.

\subsection{Predictive persistence of the curved Gaussian local limit}
\label{subsec:predictive-theorem}

We now state the dependent analogue of
Theorem~\ref{thm:two-jet-LAN}.
When nuisance parameters are present, the statement refers to the efficient
Gaussian quotient of the joint predictive experiment; the precise
least-favourable/Schur-complement reduction is given in
Supplementary Appendix~S2.

\begin{theorem}[Predictive persistence of the curved Gaussian crossover]
\label{thm:predictive-curved-LAN}
Let
\(
\varepsilon_n\downarrow0
\)
and
\(
\delta_n=\varepsilon_n^2
\)
satisfy
\begin{equation}
  \sqrt n\,\delta_n
  \longrightarrow
  \tau\in(0,\infty).
  \label{eq:predictive-critical-scale}
\end{equation}
Assume that for every compact interval
\(
K\subset\mathbb R
\)
the predictive density ratios satisfy
\eqref{eq:predictive-DQM} with a remainder
\(
r_{k,\varepsilon_n,\cdot}\in W^{1,2}(K)
\)
such that
\begin{equation}
  \frac1n
  \sum_{k=1}^n
  \mathbb E
  \left[
    \left\|
      r_{k,\varepsilon_n,\cdot}
    \right\|_{W^{1,2}(K)}^2
  \right]
  \longrightarrow
  0.
  \label{eq:predictive-uniform-remainder}
\end{equation}
and assume the conditional Lindeberg condition for the two predictive jet
coordinates.
A conditional $L^{2+\eta}$ bound for some $\eta>0$ is a convenient sufficient
condition for this Lindeberg requirement.

Suppose further that the efficient predictable Gram matrices converge,
\begin{equation}
  G_{n}^{\mathrm{pred,eff}}
  \xrightarrow{P}
  G_{\mathrm{pred}}^{\mathrm{eff}}
  =
  \begin{pmatrix}
    g_{11}^{\mathrm{pred}} &
    g_{12}^{\mathrm{pred}}\\
    g_{12}^{\mathrm{pred}} &
    g_{22}^{\mathrm{pred}}
  \end{pmatrix},
  \qquad
  G_{\mathrm{pred}}^{\mathrm{eff}}>0.
  \label{eq:predictive-efficient-limit}
\end{equation}
Then the log-likelihood ratio process, after nuisance elimination, satisfies
on every compact $K$
\begin{equation}
  \Lambda_n(\zeta)
  \Longrightarrow
  \tau Z_{\mathrm{pred}}
  (H_\zeta^\perp)
  -
  \frac{\tau^2}{2}
  \|H_\zeta^\perp\|_{\mathrm{pred}}^2,
  \qquad
  \zeta\in K,
  \label{eq:predictive-curved-limit}
\end{equation}
where
\begin{equation}
  H_\zeta^\perp
  =
  \zeta J_-^\perp
  +
  \frac{\zeta^2}{2}J_+^\perp,
  \label{eq:predictive-H-efficient}
\end{equation}
and
\(
Z_{\mathrm{pred}}
\)
is the isonormal Gaussian process on the two-dimensional Hilbert space with
Gram matrix
\(
G_{\mathrm{pred}}^{\mathrm{eff}}
\).
In particular, the limiting mean family is again a non-degenerate parabola.
\end{theorem}

\begin{proof}
Apply
Proposition~\ref{prop:appB-pred-quadratic}
to the efficient predictive jet vector.
The centering
\eqref{eq:AB-martingale}--\eqref{eq:predictive-remainder-centered},
the compact-uniform condition
\eqref{eq:predictive-uniform-remainder},
the conditional Lindeberg condition and
\eqref{eq:predictive-efficient-limit}
are exactly the hypotheses of that proposition.
It gives the quadratic log-likelihood expansion in $C(K)$ and the martingale
central-limit theorem yields
\eqref{eq:predictive-curved-limit}.
\end{proof}

The theorem shows precisely what memory changes and what it does not.
The independent Gram matrix $G$ of
Section~\ref{sec:local-asymptotics} is replaced by the efficient predictive
Gram matrix
\(
G_{\mathrm{pred}}^{\mathrm{eff}}
\).
The cross-cap image
\(
\zeta\mapsto(\zeta,\zeta^2/2)
\)
and the factor $\sqrt n$ are unchanged.

The predictive analogue of
\eqref{eq:critical-Efron-curvature} is
\begin{equation}
  \gamma_{\mathrm{pred}}(\tau)
  =
  \frac1{\tau}
  \frac{
    \sqrt{
      \det G_{\mathrm{pred}}^{\mathrm{eff}}
    }
  }{
    \left(
      g_{11}^{\mathrm{pred}}
    \right)^{3/2}
  }.
  \label{eq:predictive-curvature}
\end{equation}
Thus memory deforms the metric of the critical plane and hence its curvature,
but positivity of the predictive Gram determinant preserves genuine
two-dimensional curvature.

\subsection{CIR--OU certification}
\label{subsec:predictive-CIROU}

For the texture--coherence benchmark, the stationary predictive tangent
recursions can be solved explicitly in the strict Gaussian-texture limit.
Appendix~\ref{app:predictive-CIROU} gives the two jets
$J_{-,k}$ and $J_{+,k}$ together with the nuisance score $q_k$.
Their current-innovation chaos decomposition yields the following exact
non-degeneracy test.

\begin{proposition}[Predictive non-degeneracy in the CIR--OU benchmark]
\label{prop:predictive-CIROU-nondegenerate}
For every fixed $d,c>0$, the three directions
$q_k,J_{-,k},J_{+,k}$ are linearly independent in $L^2$.  Consequently
\begin{equation}
  \det G_{\mathrm{pred}}^{\mathrm{eff}}>0.
  \label{eq:CIR-predictive-Gram-positive}
\end{equation}
\end{proposition}

\begin{proof}
Projecting a putative relation
$aJ_{-,k}+bJ_{+,k}=s q_k$ onto the third and fourth conditional Hermite
chaoses gives a $2\times2$ system for $(a,b)$ with determinant
\begin{equation}
  -8c^3r_c^2u_dV_d\neq0
  \qquad(d,c>0),
  \label{eq:predictive-chaos-determinant}
\end{equation}
so $a=b=s=0$.  See Appendix~\ref{app:predictive-CIROU} for the explicit
recursions and projections.
\end{proof}

The remaining analytic requirement is also quantitative.  For each compact
$K\subset\mathbb R$ there is $C_{K,d,c}<\infty$ such that
\begin{equation}
  \mathbb E\!\left[
    \|r_{k,\varepsilon,\cdot}\|_{W^{1,2}(K)}^2
  \right]
  \leq C_{K,d,c}\varepsilon^2,
  \label{eq:CIR-predictive-remainder-bound}
\end{equation}
and the predictive jets have a uniform $L^{2+\eta}$ bound.  The argument is
summarized in Appendix~\ref{appC:regularity-summary} and proved in
Supplementary Appendix~S3.  Thus all assumptions of
Theorem~\ref{thm:predictive-curved-LAN} hold at fixed $d,c>0$.

For $\alpha_n/\sqrt n\to\kappa\in(0,\infty)$ one has
$\sqrt n\,\varepsilon_n^2\to\kappa^{-1}$, so the limiting predictive
experiment is the same non-degenerate Gaussian parabola as in the independent
case with the predictive efficient Gram metric.  Hence the three regimes
remain
\begin{equation}
  \alpha_n\ll\sqrt n,
  \qquad
  \alpha_n\asymp\sqrt n,
  \qquad
  \alpha_n\gg\sqrt n,
  \label{eq:predictive-three-regimes}
\end{equation}
corresponding, respectively, to the regular signed regime, the predictive
parabolic regime, and the reflection quotient.

\subsection{What memory changes---and the limits of the result}
\label{subsec:memory-interpretation}

The predictive theorem gives a precise sense in which the crossover is robust
to memory.
Conditioning on the observed past changes the two efficient tangent vectors,
their norms and their angle.
Equivalently, it replaces the one-transition Fisher plane by the pathwise
metric
\(
G_{\mathrm{pred}}^{\mathrm{eff}}
\).
It does not, however, change the representation-theoretic origin of the two
jets:
one remains the odd transverse response to weak symmetry breaking and the
other the even tangent inherited by the quotient.
Nor does memory alter the central $\sqrt n$ accumulation of information,
because predictive score increments are martingale differences.

The domain of this statement must be kept explicit.
For the CIR--OU certification above,
\begin{equation}
  d=A\Delta t>0,
  \qquad
  c=\frac{B_0\Delta t}{2}>0
  \label{eq:predictive-fixed-dc}
\end{equation}
are held fixed as
\(
n\to\infty
\)
and
\(
\alpha_n\to\infty
\).
No uniformity is claimed as
\(
d\to0
\)
or
\(
c\to0
\).
In particular, the tangent filter contains a resolvent whose correlation
length in sample units grows like
\(
d^{-1}
\);
the weak-texture expansion ceases to be uniform in the joint layer
\begin{equation}
  \alpha d=O(1).
  \label{eq:alpha-d-layer}
\end{equation}
That layer requires a separate resummation.
Likewise, the noiseless high-frequency limit is a different singular
experiment because quadratic variation can reveal diffusion parameters at a
faster information rate.
Neither regime is covered by
Theorem~\ref{thm:predictive-curved-LAN}.

Within its fixed-$(d,c)$ domain, however, the conclusion is sharp:
memory deforms the metric of the weak unfolding but does not erase its
two-jet geometry.
This completes the asymptotic side of the paper.
We now turn to a different question: which geometric data belong intrinsically
to the limiting quotient, and which belong only to its transverse unfolding?

\section{Quotient geometry versus unfolding geometry}
\label{sec:quotient-vs-unfolding}

The preceding sections have used two distinct local objects.
The limiting symmetric experiment is described intrinsically on the
identifiable quotient, whereas the weakly broken family carries an additional
transverse direction that disappears when the symmetric face is reached.
We now formalize this distinction at the level of jets.

The main conclusion is a non-determination statement in both directions:
the restricted second jet of the symmetric quotient does not retain the
transverse mixed jet, and the leading cross-cap unfolding
does not determine the higher even jets required by the intrinsic
Fisher--Amari geometry of the quotient.

\subsection{Intrinsic and transverse jet data}
\label{subsec:intrinsic-transverse-jets}

Work at a symmetric base point
\(
x_0=(\beta_0,\varepsilon,\lambda)=(\beta_0,0,0)
\).
Let
\(
E
\)
denote the one-dimensional tangent space of the external symmetry-breaking
parameter $\varepsilon$, and let
\(
N
\)
denote the one-dimensional normal space of the signed parameter $\lambda$.
The reflection acts trivially on $E$ and by the sign representation on $N$.

Let
\begin{equation}
  \mathcal H_{\mathrm{eff}}
  =
  L^2_0(P_0)/\mathcal N
  \label{eq:Heff-section7}
\end{equation}
be the efficient score space modulo nuisance tangents.
The weak-unfolding normal form of
Section~\ref{sec:weak-unfolding} provides two efficient jet classes.
The mixed transverse jet is the bilinear map
\begin{equation}
  \mathcal J_{\mathrm{mix}}:
  E\otimes N
  \longrightarrow
  \mathcal H_{\mathrm{eff}},
  \qquad
  \mathcal J_{\mathrm{mix}}
  (\partial_\varepsilon,\partial_\lambda)
  =
  [J_-],
  \label{eq:jet-mix-tensor}
\end{equation}
whereas the first invariant normal jet is
\begin{equation}
  \mathcal J_{\mathrm{quot}}:
  \operatorname{Sym}^2N
  \longrightarrow
  \mathcal H_{\mathrm{eff}},
  \qquad
  \mathcal J_{\mathrm{quot}}
  (\partial_\lambda,\partial_\lambda)
  =
  [J_+].
  \label{eq:jet-quot-tensor}
\end{equation}
The reflection parity explains the different tensor types:
\(
E\otimes N
\)
is odd in $\lambda$, while
\(
\operatorname{Sym}^2N
\)
is even.

The pair
\begin{equation}
  (\mathcal J_{\mathrm{mix}},\mathcal J_{\mathrm{quot}})
  \label{eq:jet-structural-pair}
\end{equation}
is precisely the second-order data detected by the effective cross-cap:
\(
\mathcal J_{\mathrm{mix}}
\)
describes restoration of the signed direction away from the symmetric face,
and
\(
\mathcal J_{\mathrm{quot}}
\)
describes the first direction that survives on the identifiable quotient.

\subsection{Restriction to the symmetric face loses the unfolding jet}
\label{subsec:restriction-loses-mixed-jet}

Let
\(
\mathcal J^2
\)
denote the second jet of the efficient statistical map at $x_0$.
Restriction to the symmetric face
\(
\varepsilon=0
\)
defines a linear map on second jets,
\begin{equation}
  \operatorname{Res}_0:
  J^2_{x_0}
  \longrightarrow
  J^2_{x_0}\big|_{\varepsilon=0}.
  \label{eq:jet-restriction}
\end{equation}
Any mixed coefficient containing one factor of $d\varepsilon$ lies in the
kernel of this restriction.
In particular,
\begin{equation}
  E^*\otimes N^*
  \otimes
  \mathcal H_{\mathrm{eff}}
  \subset
  \ker(\operatorname{Res}_0).
  \label{eq:mixed-jet-kernel}
\end{equation}

\begin{proposition}[Restriction erases the transverse mixed jet]
\label{prop:quotient-not-determine-unfolding}
At second order, the restriction morphism to the symmetric face annihilates
the mixed symmetry-breaking component
\(
\mathcal J_{\mathrm{mix}}
\).
Consequently, the restricted quotient $2$-jet contains no information from
which
\(
\mathcal J_{\mathrm{mix}}
\in
E^*\otimes N^*\otimes\mathcal H_{\mathrm{eff}}
\)
can be reconstructed.
\end{proposition}

\begin{proof}
The restricted jet depends only on the image of the full second jet under
\(
\operatorname{Res}_0
\).
Equation \eqref{eq:mixed-jet-kernel} shows that the entire mixed tensor sector
\(
E^*\otimes N^*\otimes\mathcal H_{\mathrm{eff}}
\)
lies in the kernel.
Hence the symmetric-face $2$-jet is insensitive to
\(
\mathcal J_{\mathrm{mix}}
\).
This proves non-determination at the level of jets.
No separate existence statement for globally distinct statistical unfoldings
is required.
\end{proof}

This proposition is coordinate independent.
For example, a smooth change
\(
\widetilde\beta
=
\beta+c\lambda^2+\cdots
\)
may modify a representative of $\mathcal J_{\mathrm{quot}}$ by a nuisance tangent, but it
cannot convert the missing
\(
E\otimes N
\)
component into data intrinsic to the restricted face.
The distinction is therefore one of tensor type, not merely one of
coordinates.

In the CIR--OU benchmark this abstract separation has a concrete meaning.
At $\varepsilon=0$, the quotient experiment retains the even direction
\(
k^\perp
\),
while the odd direction
\(
\mathfrak r^\perp
\)
that restores the sign of $\lambda$ at finite texture disappears.
The strict Gaussian quotient therefore cannot, by itself, recover the
finite-$\alpha$ coefficient multiplying
\(
\varepsilon\lambda
\).

\subsection{The cross-cap two-jet does not determine intrinsic curvature}
\label{subsec:crosscap-not-curvature}

The reciprocal limitation is equally important.
The cross-cap normal form identifies the first quotient tangent
\(
\mathcal J_{\mathrm{quot}}
\),
but intrinsic curvature of the quotient depends on higher even derivatives.

Let
\begin{equation}
  \bar\Phi(\beta,\mu)
  =
  2\sqrt{p_{\beta,\mu}}
  \label{eq:quotient-Hellinger-map-section7}
\end{equation}
be the Hellinger immersion of the identifiable quotient.
At the boundary $\mu=0$, the data carried by
\(
\mathcal J_{\mathrm{quot}}
\)
determine
\begin{equation}
  \partial_\mu\bar\Phi\big|_{\mu=0}.
  \label{eq:first-mu-derivative}
\end{equation}
By contrast, the second fundamental form and the Gaussian curvature depend on
second derivatives of the immersion, including
\begin{equation}
  \partial_{\mu\mu}^2\bar\Phi\big|_{\mu=0}.
  \label{eq:second-mu-derivative}
\end{equation}
Since
\(
\mu=\lambda^2/2
\),
the derivative
\(
\partial_{\mu\mu}^2
\)
corresponds to order $\lambda^4$ in the original signed coordinate.

\begin{proposition}[The leading unfolding does not determine quotient curvature]
\label{prop:crosscap-not-determine-curvature}
The effective cross-cap data
\(
(\mathcal J_{\mathrm{mix}},\mathcal J_{\mathrm{quot}})
\)
do not, in general, determine the intrinsic Fisher curvature of the limiting
quotient.
Additional even jet data are required.
In particular, knowledge of the first quotient tangent
\(
\partial_\mu\bar\Phi|_{\mu=0}
\)
does not determine
\(
\partial_{\mu\mu}^2\bar\Phi|_{\mu=0}
\),
and therefore does not determine the second fundamental form or the Gaussian
curvature at the identifiable boundary.
\end{proposition}

\begin{proof}
The cross-cap $2$-jet contains the mixed derivative
\(
\partial_{\varepsilon\lambda}
\)
and the first nonzero normal invariant derivative
\(
\partial_{\lambda\lambda}
\),
which becomes the first derivative
\(
\partial_\mu
\)
on the quotient.
It contains no coefficient corresponding to
\(
\partial_{\mu\mu}^2
\).
Two quotient immersions can therefore agree to first order in $\mu$ while
having different second derivatives in $\mu$.
Their Hellinger second fundamental forms may then differ, and Gauss's equation
allows the corresponding intrinsic Gaussian curvatures to differ as well.
\end{proof}

The statement concerns insufficiency of jet order.
We do not claim that every arbitrary change of
\(
\partial_{\mu\mu}^2\bar\Phi
\)
is realized by the same global statistical model.
The point is that no formula for intrinsic curvature can be extracted from
the cross-cap pair
\(
(\mathcal J_{\mathrm{mix}},\mathcal J_{\mathrm{quot}})
\)
alone.

\subsection{The benchmark requires the next even jet}
\label{subsec:benchmark-higher-even-jet}

The CIR--OU quotient makes this distinction explicit.
For the strict Gaussian-texture experiment, the first quotient derivative is
\begin{equation}
  \left.
  \partial_\mu p
  \right|_{\mu=0}
  =
  V(d)
  \left(
    D^2-D
  \right)p,
  \qquad
  D=\partial_\beta.
  \label{eq:first-quotient-density-jet}
\end{equation}
This is the density jet underlying the score
\(
J_+^\perp
\)
and the non-degenerate quotient Fisher information of
Proposition~\ref{prop:quotient-fisher-positive}.

To determine the curvature, one needs the next even jet.
Writing
\begin{equation}
  V=V(d),
  \qquad
  V_2=V(2d),
  \label{eq:V-V2}
\end{equation}
and
\begin{equation}
  K_{112}(d)
  =
  \frac{2}{d^3}
  \left[
    4d+2de^{-d}-7+8e^{-d}-e^{-2d}
  \right],
  \label{eq:K112}
\end{equation}
define
\begin{align}
  a_1
  &=
  -\frac{V_2}{4}
  +
  \frac{K_{112}}{2}
  -
  \frac{3V^2}{4},
  \\
  a_2
  &=
  \frac{V_2}{2}
  -
  \frac{3K_{112}}{2}
  +
  \frac{11V^2}{4},
  \\
  a_3
  &=
  \frac{3K_{112}}{2}
  -
  \frac{9V^2}{2}.
  \label{eq:a123-section7}
\end{align}
A direct fourth-order expansion in the signed coordinate gives
\begin{equation}
  \left.
  \partial_{\mu\mu}^2p
  \right|_{\mu=0}
  =
  8
  \left[
    a_1D
    +
    \frac{a_2}{2}D^2
    +
    \frac{a_3}{6}D^3
    +
    \frac{V^2}{8}D^4
  \right]p.
  \label{eq:second-quotient-density-jet}
\end{equation}
Thus the first quotient tangent is an order-$\lambda^2$ object, whereas the
intrinsic curvature calculation necessarily sees order $\lambda^4$ data.
In the short-texture-time limit
\(
d\to0
\),
one recovers the consistency relation
\begin{equation}
  \left.
  \partial_{\mu\mu}^2p
  \right|_{\mu=0}
  \longrightarrow
  (D^2-D)^2p.
  \label{eq:second-jet-short-time}
\end{equation}

The same hierarchy continues beyond curvature.
For the full Amari holonomy classification, the transverse
Ambrose--Singer certificate requires a third derivative in $\mu$, hence an
order-$\lambda^6$ expansion in the signed coordinate.
The corresponding computation will be used only in
Section~\ref{sec:quotient-information-geometry}; it is not part of the
cross-cap data.

\subsection{A hierarchy of geometric information}
\label{subsec:jet-hierarchy}

The preceding results can be summarized by the following hierarchy:
\begin{equation}
  \boxed{
  \begin{array}{rcl}
    \text{reflection isotropy}
    &\Longrightarrow&
    \text{allowed invariant degrees},\\[1mm]
    \mathcal J_{\mathrm{quot}}
    &\Longrightarrow&
    \text{first quotient tangent},\\[1mm]
    \mathcal J_{\mathrm{mix}}
    &\Longrightarrow&
    \text{transverse symmetry-breaking tangent},\\[1mm]
    (\mathcal J_{\mathrm{mix}},\mathcal J_{\mathrm{quot}})
    &\Longrightarrow&
    \text{cross-cap geometry and local crossover},\\[1mm]
    \text{higher even quotient jets}
    &\Longrightarrow&
    \text{intrinsic Fisher curvature and higher affine data}.
  \end{array}}
  \label{eq:geometric-information-hierarchy}
\end{equation}

This hierarchy explains why the asymptotic crossover and the intrinsic
information geometry should not be merged into a single invariant.
The former describes \emph{how the symmetric quotient is approached};
the latter describes \emph{the geometry that remains once the symmetry is
exact}.

The intrinsic branch of this hierarchy is developed next on the identifiable
quotient
\(
Q^\circ\subset Q
\),
through its Fisher metric, higher even jets, the Amari--\v{C}encov tensor and
connection holonomy.

\section{Intrinsic Fisher--Amari geometry of the limiting quotient}
\label{sec:quotient-information-geometry}

We now restrict attention to the identifiable quotient itself.
For the benchmark holonomy calculation we use the connected fast-rate domain
$\beta\in\mathbb R$ (equivalently $c>0$), so that the control points
$c=\log2$ and $c=2\log2$ lie in the same regular component.  Locally near the
reflection face, the relevant two-dimensional parameter space is
\begin{equation}
  Q
  =
  \left\{
    (\beta,\mu):
    \beta\in U_\beta,\ \mu\geq0
  \right\},
  \qquad
  Q^\circ
  =
  \left\{
    (\beta,\mu):
    \mu>0
  \right\}.
  \label{eq:quotient-Q-section8}
\end{equation}
The boundary $\mu=0$ is the image of the reflection fixed locus.
Unlike the signed parametrization $(\beta,\lambda)$, the quotient coordinates
remain statistically regular there.

This section concerns the intrinsic geometry of the full two-parameter
quotient family.
In particular, $\beta$ is \emph{not} eliminated as a nuisance coordinate in
the curvature and holonomy calculations.
Projecting it out would leave an effectively one-dimensional model and would
therefore destroy the two-dimensional intrinsic geometry studied below.
For background on the Fisher metric, the Amari--\v{C}encov tensor and
dual connections, see
\cite{Chentsov1982,AmariNagaoka2000,AyJostLeSchwachhofer2017}.

\subsection{Fisher metric up to the identifiable boundary}
\label{subsec:quotient-fisher-boundary}

Recall from Section~\ref{subsec:quotient-coordinate} that, at $\mu=0$,
the Fisher metric in the coordinates $(\beta,\mu)$ is
\begin{equation}
  g_F
  =
  \begin{pmatrix}
    I(c)
    &
    -cV(d)I(c)
    \\[1mm]
    -cV(d)I(c)
    &
    V(d)^2M_2(c)
  \end{pmatrix},
  \label{eq:boundary-fisher-section8}
\end{equation}
where
\begin{equation}
  d=A\Delta t,
  \qquad
  c=\frac{B_0\Delta t}{2},
  \qquad
  \rho=e^{-c},
  \qquad
  V(d)=\frac{2(d-1+e^{-d})}{d^2},
  \label{eq:dc-section8}
\end{equation}
and $I(c)$ and $M_2(c)$ are given in
\eqref{eq:Ic}--\eqref{eq:M2c}.
Its determinant is
\begin{equation}
  \det g_F
  =
  V(d)^2
  I(c)
  \left[
    M_2(c)-c^2I(c)
  \right]
  >0
  \qquad
  (d,c>0).
  \label{eq:boundary-fisher-determinant}
\end{equation}
Hence the Fisher metric extends non-degenerately to the identifiable
boundary.

The rescaled quotient coordinate
\begin{equation}
  \nu
  =
  V(d)\mu
  \label{eq:nu-coordinate}
\end{equation}
is useful because the boundary metric becomes
\begin{equation}
  g_F\big|_{\nu=0}
  =
  \begin{pmatrix}
    I(c)
    &
    -cI(c)
    \\[1mm]
    -cI(c)
    &
    M_2(c)
  \end{pmatrix}.
  \label{eq:boundary-fisher-nu}
\end{equation}
Thus the slow-texture factor $V(d)$ can be absorbed from the first-order
metric at the boundary.
This does not remove the slow dynamics from the statistical problem:
it only separates the magnitude of the quotient tangent from the normalized
shape of the two-dimensional Fisher geometry.

The contrast with the signed coordinates is worth recalling.
The pull-back under
\(
\mu=\lambda^2/2
\)
satisfies
\begin{equation}
  q^*g_F
  =
  g_{\beta\beta}\,d\beta^2
  +
  2\lambda g_{\beta\mu}\,d\beta\,d\lambda
  +
  \lambda^2g_{\mu\mu}\,d\lambda^2,
  \label{eq:section8-pullback}
\end{equation}
and therefore degenerates at $\lambda=0$.
The regularity of \eqref{eq:boundary-fisher-section8} is a property of the
identifiable quotient, not of the original signed parametrization.

\subsection{Higher even jets and exact Fisher curvature}
\label{subsec:quotient-curvature}

The boundary metric alone does not determine the Gaussian curvature.
As established in Section~\ref{subsec:benchmark-higher-even-jet}, one needs
the next even density jet
\(
\partial_{\mu\mu}^2p|_{\mu=0}
\),
equivalently the order-$\lambda^4$ expansion in the signed coordinate.

We use the Hellinger immersion
\begin{equation}
  \Phi(\beta,\mu)
  =
  2\sqrt{p_{\beta,\mu}}
  \label{eq:Hellinger-immersion-section8}
\end{equation}
into the radius-$2$ sphere of $L^2$.
The induced metric is exactly the Fisher metric.
If $\mathrm{II}$ denotes the second fundamental form of the quotient surface,
the Gauss equation expresses its intrinsic curvature in terms of the normal
components of the second derivatives of $\Phi$.
The order-$\lambda^4$ calculation in
\eqref{eq:second-quotient-density-jet}
therefore supplies precisely the additional information required.

The reparametrization
\(
\nu=V(d)\mu
\)
produces a substantial simplification.
After projecting the second Hellinger derivatives onto the normal space,
all dependence on the higher slow-texture coefficients
\(
V(2d)
\)
and
\(
K_{112}(d)
\)
cancels from the Gauss numerator.
The resulting boundary curvature depends only on the fast correlation
parameter
\begin{equation}
  s
  =
  \rho^2
  =
  e^{-2c}
  =
  e^{-B_0\Delta t}
  \in(0,1).
  \label{eq:s-curvature}
\end{equation}

\begin{theorem}[One-sided boundary curvature of the limiting quotient]
\label{thm:exact-quotient-curvature}
For every fixed $d>0$ and $c>0$, the Gaussian curvature of the Fisher metric admits at the identifiable boundary the one-sided value
\begin{equation}
  K_F(s)
  =
  -
  \frac{
    45s^6
    +534s^5
    +1887s^4
    +3172s^3
    +2075s^2
    +422s
    +57
  }{
    4(1+s)^2(3+s)^2(1+3s)^2
  }.
  \label{eq:KF-exact}
\end{equation}
Consequently,
\begin{equation}
  K_F(s)<0
  \qquad
  \text{for every }0<s<1.
  \label{eq:KF-negative}
\end{equation}
Moreover,
\begin{equation}
  \lim_{c\to0}K_F=-2,
  \qquad
  \lim_{c\to\infty}K_F=-\frac{19}{12}.
  \label{eq:KF-limits}
\end{equation}
In particular, the boundary curvature is independent of $d=A\Delta t$.
\end{theorem}

\begin{proof}
The order-$\lambda^4$ density jet
\eqref{eq:second-quotient-density-jet}
determines
\(
\partial_{\nu\nu}^2\Phi
\)
at the boundary.
Evaluating the normal projections in the two independent Gaussian eigenmodes
of the fast transition and inserting them into the Gauss equation yields
\eqref{eq:KF-exact}.
All coefficients in the numerator are strictly positive, while the
denominator is positive for $0<s<1$, proving
\eqref{eq:KF-negative}.
The limits follow by substituting $s\to1$ and $s\to0$, respectively.
The cancellation of all $d$-dependent terms occurs before the final rational
reduction; the detailed Hellinger calculation is given in
Appendix~\ref{app:cir-ou-formulas}.
\end{proof}

The independence of $K_F$ from $d$ should not be interpreted as independence
of the information content from the slow texture dynamics.
For example, the quotient information in the $\mu$ direction contains the
factor $V(d)^2$.
The curvature instead measures the normalized intrinsic shape of the
two-dimensional family.

By continuity, every sufficiently small interior neighbourhood of the
boundary contains points with
\(
K_F<0
\).
For the Levi--Civita connection this already determines the restricted
holonomy on the connected regular component:
\begin{equation}
  \operatorname{Hol}_0(\nabla^{(0)})
  =
  SO(2).
  \label{eq:LC-holonomy}
\end{equation}
The holonomy in \eqref{eq:LC-holonomy} is an interior statement.
The one-sided boundary jets only provide the curvature certificate from which
the nontrivial interior holonomy follows.

\subsection{The Amari--\v{C}encov tensor and the family of connections}
\label{subsec:amari-tensor}

The Fisher metric does not exhaust the local statistical geometry.
Let
\(
S_i=\partial_i\log p
\)
for the quotient coordinates
\(
\theta^1=\beta
\)
and
\(
\theta^2=\mu
\).
The symmetric cubic tensor is
\begin{equation}
  T_{ijk}
  =
  \mathbb E[
    S_iS_jS_k
  ].
  \label{eq:Amari-tensor}
\end{equation}
Together with $g_F$, it defines the Amari family of torsion-free connections
\begin{equation}
  \Gamma^{(a)k}_{ij}
  =
  \Gamma^{(0)k}_{ij}
  -
  \frac{a}{2}
  g^{k\ell}T_{ij\ell},
  \qquad
  a\in\mathbb R,
  \label{eq:a-connections}
\end{equation}
where $a=0$ is Levi--Civita and $a=\pm1$ are the exponential and mixture
connections in the convention used here.

These connections are considered on the regular interior $Q^\circ$.
Their coefficients and curvature admit one-sided limits at the boundary
because the quotient density has the required even jets.
Those limits are used as algebraic certificates for nearby interior points;
we do not define the connection holonomy \emph{at} the boundary.

A useful invariant of the cubic tensor is its metric trace
\begin{equation}
  \tau_i
  =
  g^{jk}T_{ijk}.
  \label{eq:trace-one-form}
\end{equation}
Taking the trace of the curvature of
\eqref{eq:a-connections} gives
\begin{equation}
  \operatorname{tr}R^{(a)}
  =
  -\frac{a}{2}\,d\tau.
  \label{eq:trace-curvature}
\end{equation}
Thus, whenever
\(
d\tau\neq0
\),
every nonzero $a$-connection has a non-unimodular curvature component.
For a sufficiently small contractible loop $\gamma=\partial\Sigma$ in the
regular interior,
\begin{equation}
  \det
  \operatorname{Hol}^{(a)}_\gamma
  =
  \exp
  \left[
    \frac{a}{2}
    \int_\Sigma d\tau
  \right]
  \label{eq:det-holonomy-small-loop}
\end{equation}
up to the orientation convention.
This already separates the metric connection $a=0$ from the generic
non-metric members of the family, but it does not by itself prove full
$GL^+(2,\mathbb R)$ holonomy.

\subsection{Ambrose--Singer certificate and restricted holonomy}
\label{subsec:holonomy-classification}

We first make explicit how one-sided boundary jets are used in an interior
holonomy argument.

\begin{lemma}[Persistence of a boundary certificate]
\label{lem:boundary-to-interior}
Let $\mathcal D(\beta,\mu)$ be any scalar discriminant built smoothly from the
coefficients of a connection, its curvature and finitely many covariant
derivatives on $Q^\circ$, and assume that these quantities admit continuous
one-sided extensions to $\mu=0$.
If
\begin{equation}
  \mathcal D(\beta_0,0)\neq0,
  \label{eq:boundary-discriminant-nonzero}
\end{equation}
then there exists $\mu_*>0$, arbitrarily small, such that
\begin{equation}
  \mathcal D(\beta_0,\mu_*)\neq0.
  \label{eq:interior-discriminant-nonzero}
\end{equation}
Any Ambrose--Singer generation criterion certified by
$\mathcal D\neq0$ is therefore realized at a regular interior point.
\end{lemma}

\begin{proof}
This follows directly from continuity of the one-sided extension.
\end{proof}

The remaining step uses the Ambrose--Singer theorem
\cite{AmbroseSinger1953}, which places parallel transports of curvature
endomorphisms in the holonomy algebra.  Differentiating a transported
curvature endomorphism along a smooth interior curve shows that the
corresponding covariant curvature derivative belongs to the same
finite-dimensional holonomy algebra; no analyticity assumption is needed.

We use the following elementary two-dimensional criterion.
If
\(
X,Y\in\mathfrak{gl}(2,\mathbb R)
\)
satisfy
\begin{equation}
  \operatorname{tr}X\neq0,
  \qquad
  \det[X,Y]\neq0,
  \label{eq:gl2-generation-criterion}
\end{equation}
then
\begin{equation}
  \operatorname{Lie}\{X,Y\}
  =
  \mathfrak{gl}(2,\mathbb R).
  \label{eq:gl2-generation}
\end{equation}
Indeed, the nonsingular commutator rules out a solvable triangularizable
subalgebra, while the nonzero trace supplies the scalar direction in addition
to the $\mathfrak{sl}(2,\mathbb R)$ part.

For the exponential connection $a=1$, let
\begin{equation}
  \mathsf R_1
  =
  R^{(1)}
  (\partial_\beta,\partial_\nu),
  \qquad
  \mathsf D_{\nu,1}
  =
  \left(
    \nabla^{(1)}_{\partial_\nu}
    R^{(1)}
  \right)
  (\partial_\beta,\partial_\nu).
  \label{eq:exp-holonomy-generators}
\end{equation}
At the control point
\(
c=\log2
\),
the one-sided boundary value obtained from the exact quotient jets satisfies
\begin{equation}
  \left.
  \det[\mathsf R_1,\mathsf D_{\nu,1}]
  \right|_{\mu=0}
  >0
  \qquad
  \text{for every }d>0.
  \label{eq:exp-discriminant-positive}
\end{equation}
By Lemma~\ref{lem:boundary-to-interior}, the same discriminant is nonzero at
regular interior points with $\mu>0$ arbitrarily close to the boundary.
The proof reduces all slow-dynamics dependence to
\begin{equation}
  \mathcal Q(d)
  =
  \frac{K_{112}(d)}{V(d)^2},
  \label{eq:Q-holonomy}
\end{equation}
for which
\begin{equation}
  2<\mathcal Q(d)<\frac94
  \qquad
  (d>0).
  \label{eq:Q-holonomy-bound}
\end{equation}
The resulting discriminant factorization is strictly positive on this entire
interval.
Hence
\begin{equation}
  \mathfrak{hol}_0(\nabla^{(1)})
  =
  \mathfrak{gl}(2,\mathbb R),
  \qquad
  \operatorname{Hol}_0(\nabla^{(1)})
  =
  GL^+(2,\mathbb R).
  \label{eq:exp-holonomy}
\end{equation}
Duality with respect to the Fisher metric gives the same restricted holonomy
for the mixture connection $a=-1$.

The values $a=\pm1$ require the transverse certificate above because a
particular longitudinal discriminant vanishes there.
For the remaining nonzero values of $a$, two longitudinal control points,
\begin{equation}
  c=\log2,
  \qquad
  c=2\log2,
  \label{eq:two-holonomy-control-points}
\end{equation}
are sufficient.
After removal of the trivial factors
\(
a^2(a^2-1)
\),
the two commutator discriminants reduce to cubic polynomials in
\(
a^2
\)
whose coefficients depend on $d$ only through
\(
\mathcal Q(d)
\).
Their resultant is strictly positive for
\(
2\leq\mathcal Q\leq9/4
\).
Consequently the two boundary discriminants cannot vanish simultaneously:
for every fixed $a\neq0,\pm1$, at least one of the two control points has a
nonzero one-sided boundary discriminant.
Lemma~\ref{lem:boundary-to-interior} then provides a nearby regular interior
point at which the generation criterion
\eqref{eq:gl2-generation-criterion} holds.
Supplementary Appendix~S4 records the exact transverse factorization, the
resultant, and the Bernstein positivity certificate.

\begin{theorem}[Restricted holonomy of the quotient Amari family]
\label{thm:full-amari-holonomy}
Fix $d=A\Delta t>0$.
On the connected regular component of the Gaussian-texture quotient
containing the interior points adjacent to $\mu=0$,
\begin{equation}
  \operatorname{Hol}_0
  \left(
    \nabla^{(a)}
  \right)
  =
  \begin{cases}
    SO(2),
    &
    a=0,
    \\[1mm]
    GL^+(2,\mathbb R),
    &
    a\neq0.
  \end{cases}
  \label{eq:full-amari-holonomy}
\end{equation}
Thus Levi--Civita is the unique exceptional connection in the full
Amari family for this quotient model.
\end{theorem}

\begin{proof}
For $a=0$, strict negative Fisher curvature at interior points arbitrarily
close to the boundary gives the holonomy algebra
\(
\mathfrak{so}(2)
\),
hence restricted holonomy $SO(2)$.

For $a=\pm1$, the transverse boundary certificate at $c=\log2$ persists,
by Lemma~\ref{lem:boundary-to-interior}, at a nearby interior point and generates
\(
\mathfrak{gl}(2,\mathbb R)
\).
For $a\neq0,\pm1$, the two-point resultant certificate guarantees that at
least one of the boundary control points in
\eqref{eq:two-holonomy-control-points}
has a nonzero boundary discriminant; continuity gives a nearby interior point
with curvature generators satisfying
\eqref{eq:gl2-generation-criterion}.
Thus the holonomy algebra is
\(
\mathfrak{gl}(2,\mathbb R)
\).
Restricted holonomy is connected, so the corresponding group is
\(
GL^+(2,\mathbb R)
\).
Finally, holonomy algebras at two points of the same connected regular
component are conjugate by parallel transport, which extends the
classification from the control points to the whole component.
\end{proof}

\subsection{Hierarchy of intrinsic information}
\label{subsec:intrinsic-hierarchy}

The quotient geometry obtained above is stratified by jet order.
At the boundary,
\begin{equation}
  \partial_\mu p
  \quad
  \text{determines the first identifiable tangent and the boundary Fisher
  metric},
  \label{eq:hierarchy-first}
\end{equation}
while
\begin{equation}
  \partial_{\mu\mu}^2p
  \quad
  \text{enters the second fundamental form and the Fisher curvature}.
  \label{eq:hierarchy-second}
\end{equation}
The complete Ambrose--Singer certificate for the non-metric connections
requires one further transverse derivative, and hence the third quotient jet
\begin{equation}
  \partial_{\mu\mu\mu}^3p,
  \label{eq:hierarchy-third}
\end{equation}
which corresponds to order $\lambda^6$ in the original signed coordinate.

This yields the intrinsic sequence
\begin{equation}
  \boxed{
  \text{first quotient jet}
  \longrightarrow
  g_F
  \longrightarrow
  \text{higher even jets}
  \longrightarrow
  K_F
  \longrightarrow
  (g_F,T)
  \longrightarrow
  \nabla^{(a)}
  \longrightarrow
  \operatorname{Hol}_0(\nabla^{(a)}).
  }
  \label{eq:intrinsic-sequence}
\end{equation}
It should be kept separate from the transverse sequence developed in
Sections~\ref{sec:weak-unfolding}--\ref{sec:predictive}.
The cross-cap and its curved LAN limit describe how the quotient is
\emph{approached}; the Fisher--Amari geometry describes what remains
\emph{inside} the exact quotient once the transverse symmetry-breaking
direction has disappeared.

The two layers together complete the geometric picture of the emergent
symmetry transition.

\section{Conclusion}
\label{sec:conclusion}

We have studied a transition in which a regular observed experiment approaches
an exact reflection quotient.  In the symmetric limit the signed parameter
$\lambda$ loses its first-order score, but the invariant coordinate
$\mu=\lambda^2/2$ remains regular.  The singularity is therefore generated by
an emergent observed symmetry rather than imposed in the original parameter
space.

At the observed Hellinger level, reflection symmetry forces the leading local
displacement to split into a mixed sign-restoring jet and an even quotient
jet, \(\varepsilon\lambda J_-+\lambda^2J_+/2\).
After nuisance elimination, their independence gives a nondegenerate
cross-cap two-jet.  For $n$ observations the corresponding resolution map is
controlled by $\tau_n=\sqrt n\,\varepsilon_n^2$, producing a regular signed
regime, a critical curved Gaussian parabola, and a quotient regime in which
the familiar $n^{-1/4}$ signed scale is simply the pullback of the regular
$n^{-1/2}$ scale in $\mu$.  Predictive factorization shows that the same
critical geometry survives under memory: dependence changes the efficient
Gram metric but not the parity structure or the central $\sqrt n$ scale.

The quotient and its weak unfolding nevertheless carry different information.
Restriction to the symmetric face removes the mixed jet, whereas the
cross-cap two-jet does not determine the higher even derivatives of the
quotient family.  In the CIR--OU benchmark this hierarchy can be followed
exactly: the order-$\lambda^2$ jet determines the first quotient tangent,
order $\lambda^4$ determines the intrinsic Fisher curvature, and order
$\lambda^6$ enters the complete Ambrose--Singer certificate.

On the identifiable quotient the Fisher metric is nondegenerate up to the
boundary, with a strictly negative one-sided boundary curvature.  The nearby
regular interior consequently has nontrivial Levi--Civita holonomy.
For the full Amari family,
\(\operatorname{Hol}_0(\nabla^{(a)})=SO(2)\) at \(a=0\), whereas
\(\operatorname{Hol}_0(\nabla^{(a)})=GL^+(2,\mathbb R)\) for \(a\neq0\).
The extrinsic geometry of the critical Gaussian unfolding and the intrinsic
Fisher--Amari geometry of the exact quotient are therefore complementary
layers of the same symmetry transition, not alternative descriptions of one
object.

The present analysis is intentionally local and one-reflection-dimensional.
The predictive theorem keeps $d=A\Delta t$ and $c=B_0\Delta t/2$ fixed and
does not cover joint high-frequency limits or the non-uniform layer
$\alpha d=O(1)$.  Higher-dimensional isotropy representations, several weakly
broken directions, and quotients with changing stabilizer type provide natural
extensions in which higher-dimensional critical Gaussian submanifolds and a
richer invariant-jet hierarchy may appear.  The mechanism isolated here gives
a concrete template for testing those extensions: an emergent symmetry selects
the quotient jets, while weak symmetry breaking supplies transverse jets whose
statistical resolution can remain visible before the quotient is reached.

\appendix
\section{Observed Hellinger regularity and latent sufficient criteria}
\label{app:hellinger-normal-form}

This appendix separates three logically distinct ingredients used in
Section~\ref{sec:weak-unfolding}:
the Hilbert-valued Hellinger Taylor expansion, a sufficient condition for
rewriting it as a normalized square-root likelihood ratio, and the stronger
latent operator criterion that explains the parity of the two jets in a
restricted class of hidden-variable models.

\subsection{Hilbert-valued two-jet expansion}
\label{app:observed-Hellinger-Taylor}

Let \(\Phi(\varepsilon,\lambda)=2\sqrt{p_{\varepsilon,\lambda}}\in L^2(\varpi)\)
be $C^2$ near the origin and suppose
\(
\Phi(0,\lambda)=\Phi(0,-\lambda)
\).
Then
\(
\partial_\lambda\Phi(0,0)=0
\).
For completeness, the integral remainder used in
Theorem~\ref{thm:weak-breaking-normal-form} may be written
\begin{equation}
\begin{split}
&
\Phi(\varepsilon,\lambda)
-
\Phi(\varepsilon,0)
-
\varepsilon\lambda
\partial_{\varepsilon\lambda}^2\Phi(0,0)
-
\frac{\lambda^2}{2}
\partial_{\lambda\lambda}^2\Phi(0,0)
\\
&\qquad
=
\lambda
\int_0^1
\Big[
  \partial_\lambda\Phi(\varepsilon,t\lambda)
  -
  \varepsilon
  \partial_{\varepsilon\lambda}^2\Phi(0,0)
  -
  t\lambda
  \partial_{\lambda\lambda}^2\Phi(0,0)
\Big]
\,dt.
\label{eq:appA-Hilbert-remainder}
\end{split}
\end{equation}
Continuity of the derivative of
\(
\partial_\lambda\Phi
\)
therefore gives
\begin{equation}
  \left\|
    \text{remainder in \eqref{eq:appA-Hilbert-remainder}}
  \right\|_{L^2(\varpi)}
  =
  o(
    \varepsilon|\lambda|+\lambda^2
  ).
  \label{eq:appA-Hilbert-remainder-bound}
\end{equation}
This proof does not refer to a latent representation.

The normalization
\(
\|\Phi(\varepsilon,\lambda)\|_2^2=4
\)
also gives
\begin{equation}
  \left\langle
    \partial_{\varepsilon\lambda}^2\Phi(0,0),
    \sqrt{p_0}
  \right\rangle
  =
  0,
  \qquad
  \left\langle
    \partial_{\lambda\lambda}^2\Phi(0,0),
    \sqrt{p_0}
  \right\rangle
  =
  0,
  \label{eq:appA-tangent-orthogonality}
\end{equation}
because
\(
\partial_\lambda\Phi(0,0)=0
\).
Thus, when $p_0>0$ on the relevant support, the two Hellinger jets are
represented by centered scores
\(
J_-,J_+\in L_0^2(P_0)
\).

\subsection{From ambient Hellinger displacement to a normalized ratio}
\label{app:normalized-Hellinger-ratio}

The geometric normal form is naturally expressed through
\(
2\sqrt{p_{\varepsilon,\lambda}}
-
2\sqrt{p_{\varepsilon,0}}
\).
For triangular-array likelihood theory one often wants instead
\(
\sqrt{p_{\varepsilon,\lambda}/p_{\varepsilon,0}}
\).
The following is only a sufficient comparison criterion.

Set
\begin{equation}
  b_\varepsilon
  =
  \frac{p_{\varepsilon,0}}{p_0}.
  \label{eq:appA-baseline-ratio}
\end{equation}

\begin{proposition}[A sufficient baseline comparison]
\label{prop:appA-baseline-comparison}
Assume
\begin{equation}
  0<m
  \leq
  b_\varepsilon
  \leq
  M<\infty
  \qquad
  P_0\text{-a.s.},
  \label{eq:appA-baseline-uniform-bounds}
\end{equation}
for all sufficiently small $\varepsilon$, and
\begin{equation}
  \|b_\varepsilon-1\|_{L^\infty(P_0)}
  \longrightarrow0.
  \label{eq:appA-baseline-Linf}
\end{equation}
If the ambient Hellinger normal form
\begin{equation}
\begin{split}
  2\sqrt{p_{\varepsilon,\lambda}}
  -
  2\sqrt{p_{\varepsilon,0}}
  ={}&
  \sqrt{p_0}
  \left(
    \varepsilon\lambda J_-
    +
    \frac{\lambda^2}{2}J_+
  \right)
\\
&+
o_{L^2(\varpi)}
\left(
  \varepsilon|\lambda|+\lambda^2
\right)
\end{split}
  \label{eq:appA-ambient-form}
\end{equation}
holds, then
\begin{equation}
  \sqrt{
    \frac{p_{\varepsilon,\lambda}}
         {p_{\varepsilon,0}}
  }
  =
  1
  +
  \frac12
  \left(
    \varepsilon\lambda J_-
    +
    \frac{\lambda^2}{2}J_+
  \right)
  +
  o_{L^2(P_0)}
  \left(
    \varepsilon|\lambda|+\lambda^2
  \right).
  \label{eq:appA-normalized-form}
\end{equation}
\end{proposition}

\begin{proof}
Divide
\eqref{eq:appA-ambient-form}
by
\(
2\sqrt{p_{\varepsilon,0}}
=
2\sqrt{p_0}\sqrt{b_\varepsilon}
\).
The multiplier
\(
b_\varepsilon^{-1/2}
\)
is uniformly bounded and converges to one in $L^\infty(P_0)$ by
\eqref{eq:appA-baseline-uniform-bounds}--\eqref{eq:appA-baseline-Linf}.
Multiplication therefore preserves the stated $L^2$ remainder and its leading
coefficient.
\end{proof}

Condition
\eqref{eq:appA-baseline-Linf}
is deliberately stronger than necessary.
It is useful as a transparent sufficient criterion, but the benchmark in this
paper is not certified through this route.
The independent and predictive LAN theorems instead assume the required
normalized Hellinger or density-ratio expansion directly.

\paragraph{Latent sufficient mechanism.}
A stronger hidden-variable criterion identifies the two parity-selected jets when
all weak-symmetry-breaking dependence can be placed in an $L^2$ perturbation of
the latent law and the observation kernel admits a second-order operator
expansion.  The statement and quantitative remainder estimate are recorded in
Supplementary Appendix~S1; they are not required for the CIR--OU benchmark.

\subsection{Stability of nuisance projection}
\label{app:nuisance-projection-stability}

Let the finite-dimensional nuisance scores under $P_{\varepsilon,0}$ be
transported isometrically to $L^2(P_0)$ and suppose that they converge there
to a linearly independent limiting score family.
If $\Pi_\varepsilon$ and $\Pi_0$ denote the corresponding orthogonal
projectors, then
\begin{equation}
  \|\Pi_\varepsilon-\Pi_0\|_{\rm op}\longrightarrow0.
  \label{eq:appA-projector-convergence}
\end{equation}
Consequently, a score expansion
$\widetilde S_\varepsilon=\varepsilon J+o_{L^2}(\varepsilon)$ has efficient
part
\begin{equation}
  (I-\Pi_\varepsilon)\widetilde S_\varepsilon
  =
  \varepsilon(I-\Pi_0)J+o_{L^2}(\varepsilon).
  \label{eq:appA-efficient-leading}
\end{equation}
\begin{proposition}[Continuity of the nuisance projector]
\label{prop:appA-projector-stability}
The preceding conclusions hold whenever the transported nuisance scores
converge in $L^2(P_0)$ and their limiting Gram matrix is positive definite.
\end{proposition}
The finite-dimensional operator proof, including the transported-score
construction, is given in Supplementary Section~S1.

\subsection{How the CIR--OU benchmark enters the theorem}
\label{app:benchmark-Hellinger-checklist}

The CIR--OU benchmark is not certified by
the quantitative latent criterion in Supplementary Appendix~S1.
At finite texture variance the control parameter
\(
\varepsilon=\alpha^{-1/2}
\)
changes the standardized CIR dynamics and the coherent observation map, not
only the stationary latent law.
Moreover, the stationary Gamma-to-Gaussian expansion is naturally a local or
moment expansion rather than a global $L^2(\nu_0)$ Radon--Nikodym perturbation.

The benchmark instead uses the observed fixed-horizon DQM calculation.
Conditioned on a standardized texture trajectory,
the observed vector is Gaussian, and differentiation of the conditional
Gaussian likelihood followed by projection onto the observed sigma-field
gives
\begin{equation}
  S_{\lambda,\mathrm{obs}}
  =
  \varepsilon J_-
  +
  o_{L^2}(\varepsilon)
  \label{eq:appA-benchmark-odd-DQM}
\end{equation}
at $\lambda=0$.
The strict Gaussian-texture quotient calculation independently gives
\begin{equation}
  S_\mu=J_+.
  \label{eq:appA-benchmark-even-DQM}
\end{equation}
Their explicit one-transition representatives and Gram matrix are derived in
Appendix~\ref{app:cir-ou-formulas}.
Thus the benchmark verifies the observed two-jet directly, while the latent
criterion above remains a structural sufficient mechanism for a narrower
class of models.

\section{Triangular arrays and predictive Gaussian limits}
\label{app:local-asymptotics}

This appendix supplies the likelihood expansions used in
Sections~\ref{sec:local-asymptotics} and \ref{sec:predictive}.
Two points deserve to be made explicit.
First, a Hellinger expansion determines the full quadratic drift of the
log-likelihood only after the normalization of the density ratio has been
used.
Second, when nuisance parameters are present, the efficient jets arise from
the Gaussian quotient of the joint local experiment, or equivalently from a
least-favourable local nuisance adjustment.
The notation ``after nuisance elimination'' in the main text refers to this
standard reduction.

\subsection{Independent triangular arrays}
\label{app:iid-triangular-arrays}

For each $n$, let $P_n=P_{\varepsilon_n,0}$ and let
\(
V_n:\mathsf Y\to\mathbb R^m
\)
be centered under $P_n$:
\begin{equation}
  \mathbb E_n[V_n]=0.
  \label{eq:appB-centered-V}
\end{equation}
In the applications of the main text, $m=2$ and
\(
V_n=(J_{-,\varepsilon_n},J_{+,\varepsilon_n})^\top
\),
or its efficient version after the nuisance reduction described below.
Let
\begin{equation}
  G_n
  =
  \mathbb E_n[V_nV_n^\top]
  \longrightarrow
  G.
  \label{eq:appB-Gn}
\end{equation}

Consider a sequence of local coordinates
\(
v_n\in\mathbb R^m
\)
such that
\begin{equation}
  u_n
  =
  \sqrt n\,v_n
  \longrightarrow
  u.
  \label{eq:appB-un}
\end{equation}
Assume a one-observation Hellinger expansion
\begin{equation}
  \sqrt{\frac{dP_{n,v_n}}{dP_n}}
  =
  1
  +
  \frac12 v_n^\top V_n
  +
  r_n,
  \label{eq:appB-Hellinger}
\end{equation}
where
\begin{equation}
  n\,\mathbb E_n[r_n^2]
  \longrightarrow0.
  \label{eq:appB-r-L2}
\end{equation}
All random variables below are taken under the baseline law $P_n$.

\subsubsection{Normalization of the Hellinger remainder}

Set
\begin{equation}
  w_n
  =
  \frac12v_n^\top V_n+r_n.
  \label{eq:appB-wn}
\end{equation}
Since the density ratio is normalized,
\begin{equation}
  \mathbb E_n[(1+w_n)^2]=1.
  \label{eq:appB-normalization}
\end{equation}
Using \eqref{eq:appB-centered-V},
\begin{equation}
  2\mathbb E_n[r_n]
  +
  \mathbb E_n[w_n^2]
  =
  0.
  \label{eq:appB-Er-exact}
\end{equation}
Thus the mean of the apparently negligible Hellinger remainder carries half
of the deterministic LAN drift.

\begin{lemma}[Mean of the Hellinger remainder]
\label{lem:appB-remainder-mean}
Under
\eqref{eq:appB-Gn}--\eqref{eq:appB-r-L2},
\begin{equation}
  \mathbb E_n[r_n]
  =
  -
  \frac18
  v_n^\top G_n v_n
  +
  o(n^{-1}).
  \label{eq:appB-Er-asymptotic}
\end{equation}
\end{lemma}

\begin{proof}
By Cauchy--Schwarz,
\[
  \left|
    \mathbb E_n[
      (v_n^\top V_n)r_n
    ]
  \right|
  \leq
  (v_n^\top G_nv_n)^{1/2}
  \|r_n\|_2
  =
  o(n^{-1}),
\]
because
\(
\|v_n\|=O(n^{-1/2})
\)
and
\(
\|r_n\|_2=o(n^{-1/2})
\).
Also
\(
\mathbb E_n[r_n^2]=o(n^{-1})
\).
Hence
\(\mathbb E_n[w_n^2]=\frac14v_n^\top G_nv_n+o(n^{-1})\).
Equation \eqref{eq:appB-Er-exact} gives
\eqref{eq:appB-Er-asymptotic}.
\end{proof}

\subsubsection{Quadratic log-likelihood expansion}

Let
\(
V_{n,1},\ldots,V_{n,n}
\)
and
\(
r_{n,1},\ldots,r_{n,n}
\)
be independent copies.
Define
\begin{equation}
  \Delta_n
  =
  \frac1{\sqrt n}
  \sum_{i=1}^nV_{n,i},
  \qquad
  \widehat G_n
  =
  \frac1n
  \sum_{i=1}^n
  V_{n,i}V_{n,i}^\top.
  \label{eq:appB-Delta-Ghat}
\end{equation}

We assume the vector Lindeberg condition
\begin{equation}
  \mathbb E_n
  \left[
    \|V_n\|^2
    \mathbf 1_{
      \{\|V_n\|>\eta\sqrt n\}
    }
  \right]
  \longrightarrow0
  \qquad
  \text{for every }\eta>0.
  \label{eq:appB-Lindeberg}
\end{equation}
A uniform $L^{2+\delta}$ bound is sufficient.

\begin{lemma}[Quadratic likelihood expansion]
\label{lem:appB-quadratic-LR}
Assume
\eqref{eq:appB-Gn},
\eqref{eq:appB-un},
\eqref{eq:appB-r-L2}
and
\eqref{eq:appB-Lindeberg}.
Then
\begin{equation}
  \log
  \frac{
    dP_{n,v_n}^{\otimes n}
  }{
    dP_n^{\otimes n}
  }
  =
  u_n^\top\Delta_n
  -
  \frac12
  u_n^\top G_nu_n
  +
  o_{P_n^{\otimes n}}(1).
  \label{eq:appB-LAN-expansion}
\end{equation}
Equivalently, one may replace $G_n$ by $\widehat G_n$ in the quadratic term.
\end{lemma}

\begin{proof}
For the $i$th observation set
\(w_{n,i}=\frac12v_n^\top V_{n,i}+r_{n,i}\).
The log-density ratio is
\begin{equation}
  2\log(1+w_{n,i}).
  \label{eq:appB-log-one}
\end{equation}
The Lindeberg condition implies
\(\max_{1\leq i\leq n}|v_n^\top V_{n,i}|\xrightarrow{P}0\),
and
\eqref{eq:appB-r-L2} gives
\(\max_{1\leq i\leq n}|r_{n,i}|\xrightarrow{P}0\)
by the union bound.
Hence
\(
\max_i|w_{n,i}|\to_P0
\).

On the event
\(
\max_i|w_{n,i}|\leq1/2
\),
Taylor's formula gives
\begin{equation}
  2\log(1+w)
  =
  2w-w^2+\mathcal R(w),
  \qquad
  |\mathcal R(w)|
  \leq
  C|w|^3.
  \label{eq:appB-log-Taylor}
\end{equation}
Moreover,
\[
  \sum_{i=1}^nw_{n,i}^2=O_P(1),
\]
because
\(
n\|v_n\|^2=O(1)
\)
and
\(
n\mathbb E_n[r_n^2]\to0
\).
Therefore
\[
  \sum_{i=1}^n
  |\mathcal R(w_{n,i})|
  \leq
  C
  \max_i|w_{n,i}|
  \sum_{i=1}^nw_{n,i}^2
  =
  o_P(1).
\]

For the linear term,
\[
  2\sum_{i=1}^nw_{n,i}
  =
  v_n^\top
  \sum_{i=1}^nV_{n,i}
  +
  2\sum_{i=1}^n
  (r_{n,i}-\mathbb E_n r_n)
  +
  2n\mathbb E_n r_n.
\]
The centered remainder sum is $o_P(1)$ because its variance is bounded by
\(
4n\mathbb E_n[r_n^2]\to0
\).
By Lemma~\ref{lem:appB-remainder-mean},
\begin{equation}
  2n\mathbb E_n r_n
  =
  -
  \frac14
  n v_n^\top G_nv_n
  +
  o(1).
  \label{eq:appB-drift-half-one}
\end{equation}

For the quadratic term,
\[
  \sum_{i=1}^nw_{n,i}^2
  =
  \frac14
  v_n^\top
  \left(
    \sum_{i=1}^n
    V_{n,i}V_{n,i}^\top
  \right)
  v_n
  +
  o_P(1),
\]
the terms involving $r_{n,i}$ being negligible by Cauchy--Schwarz and
\eqref{eq:appB-r-L2}.
The triangular-array law of large numbers implied by
\eqref{eq:appB-Lindeberg} yields
\begin{equation}
  u_n^\top
  (\widehat G_n-G_n)
  u_n
  \xrightarrow{P}0.
  \label{eq:appB-Gram-LLN}
\end{equation}
Thus the quadratic Taylor term contributes
\(-\frac14u_n^\top G_nu_n+o_P(1)\).
Together with
\eqref{eq:appB-drift-half-one},
this gives the full coefficient
\(
-1/2
\)
in \eqref{eq:appB-LAN-expansion}.
\end{proof}

The vector Lindeberg--Feller theorem (see, e.g., \cite{vanDerVaart1998,LeCamYang2000}) gives
\begin{equation}
  \Delta_n
  \Longrightarrow
  Z,
  \qquad
  Z\sim N(0,G).
  \label{eq:appB-vector-CLT}
\end{equation}
Lemma~\ref{lem:appB-quadratic-LR} therefore proves
Theorem~\ref{thm:two-jet-LAN}.

\subsection{Compact-uniform critical experiment}
\label{app:iid-process-limit}

For the critical layer, set
\begin{equation}
  b(\zeta)
  =
  \begin{pmatrix}
    \zeta\\
    \zeta^2/2
  \end{pmatrix},
  \qquad
  v_n(\zeta)
  =
  \delta_n b(\zeta),
  \qquad
  \sqrt n\,\delta_n\to\tau.
  \label{eq:appB-critical-b}
\end{equation}
Pointwise convergence follows directly from the preceding result.
For process convergence on a compact interval
\(
K\subset\mathbb R
\),
we use a slightly stronger uniform remainder condition.

Let
\(
r_n(\cdot)
\)
be a random element of the Sobolev space
\(
W^{1,2}(K)
\).
Assume
\begin{equation}
  n\,
  \mathbb E_n
  \left[
    \|r_n\|_{W^{1,2}(K)}^2
  \right]
  \longrightarrow0.
  \label{eq:appB-iid-Sobolev-remainder}
\end{equation}
In one dimension,
\begin{equation}
  \|f\|_{C(K)}
  \leq
  C_K
  \|f\|_{W^{1,2}(K)}.
  \label{eq:appB-Sobolev-embedding}
\end{equation}
Hence
\eqref{eq:appB-iid-Sobolev-remainder}
makes the accumulated Hellinger remainder uniformly negligible on $K$.

\begin{proposition}[Critical convergence in $C(K)$]
\label{prop:appB-iid-process}
Assume
\eqref{eq:appB-Gn},
\eqref{eq:appB-Lindeberg},
\eqref{eq:appB-critical-b}
and the compact-uniform Hellinger expansion with
\eqref{eq:appB-iid-Sobolev-remainder}.
Then, in $C(K)$,
\begin{equation}
  \Lambda_n(\zeta)
  \Longrightarrow
  \tau\,b(\zeta)^\top Z
  -
  \frac{\tau^2}{2}
  b(\zeta)^\top G b(\zeta),
  \qquad
  Z\sim N(0,G).
  \label{eq:appB-iid-process-limit}
\end{equation}
\end{proposition}

\begin{proof}
The leading random term is
\((\sqrt n\,\delta_n)b(\zeta)^\top\Delta_n\),
and the quadratic term is
\(-\frac12n\delta_n^2b(\zeta)^\top\widehat G_nb(\zeta)\).
Both are continuous polynomial functions of $\zeta$ whose random
coefficients converge.
The Sobolev remainder condition makes the difference between the exact
log-likelihood process and this polynomial approximation
$o_P(1)$ in the supremum norm.
The continuous mapping theorem gives
\eqref{eq:appB-iid-process-limit}.
\end{proof}

For
\(
b(\zeta)
=
(\zeta,\zeta^2/2)^\top
\),
the limit in
\eqref{eq:appB-iid-process-limit}
is precisely the curved Gaussian experiment
\eqref{eq:critical-curved-LAN}.

\subsection{Predictive martingale arrays}
\label{app:predictive-martingale-LAN}

We now turn to one dependent trajectory.
Let
\(
(\mathcal F_{n,k})_{0\leq k\leq n}
\)
be the observed filtration under the baseline model and let
\begin{equation}
  V_{n,k}
  =
  \begin{pmatrix}
    J_{-,n,k}\\
    J_{+,n,k}
  \end{pmatrix}
  \label{eq:appB-pred-V}
\end{equation}
be a two-dimensional martingale-difference array:
\begin{equation}
  \mathbb E[
    V_{n,k}
    \mid
    \mathcal F_{n,k-1}
  ]
  =
  0.
  \label{eq:appB-pred-MD}
\end{equation}
After nuisance reduction, the same notation may be used for the efficient
array, with limiting Gram matrix
\(
G_{\rm pred}^{\rm eff}
\).

Let
\begin{equation}
  G_n^{\rm pred}
  =
  \frac1n
  \sum_{k=1}^n
  \mathbb E[
    V_{n,k}V_{n,k}^\top
    \mid
    \mathcal F_{n,k-1}
  ].
  \label{eq:appB-pred-Gram}
\end{equation}
Assume
\begin{equation}
  G_n^{\rm pred}
  \xrightarrow{P}
  G_{\rm pred}
  \label{eq:appB-pred-Gram-limit}
\end{equation}
and the conditional Lindeberg condition
\begin{equation}
  \frac1n
  \sum_{k=1}^n
  \mathbb E
  \left[
    \|V_{n,k}\|^2
    \mathbf 1_{\{\|V_{n,k}\|>\eta\sqrt n\}}
    \mid
    \mathcal F_{n,k-1}
  \right]
  \xrightarrow{P}
  0
  \label{eq:appB-pred-Lindeberg}
\end{equation}
for every $\eta>0$.

The vector martingale central limit theorem \cite{HallHeyde1980} then yields
\begin{equation}
  \Delta_n^{\rm pred}
  =
  \frac1{\sqrt n}
  \sum_{k=1}^nV_{n,k}
  \Longrightarrow
  N(0,G_{\rm pred}).
  \label{eq:appB-pred-CLT}
\end{equation}

\subsubsection{Predictive density-ratio remainder}

Fix a compact interval $K$ and set
\begin{equation}
  b(\zeta)
  =
  \begin{pmatrix}
    \zeta\\
    \zeta^2/2
  \end{pmatrix}.
  \label{eq:appB-pred-b}
\end{equation}
Let
\(
\delta_n\downarrow0
\)
with
\begin{equation}
  \sqrt n\,\delta_n\to\tau\in(0,\infty).
  \label{eq:appB-pred-delta}
\end{equation}
Assume the conditional density ratio has the expansion
\begin{equation}
  R_{n,k}(\zeta)
  =
  1
  +
  \delta_n
  \left[
    b(\zeta)^\top V_{n,k}
    +
    r_{n,k}(\zeta)
  \right].
  \label{eq:appB-pred-ratio}
\end{equation}
Because $R_{n,k}(\zeta)$ is a conditional density ratio,
\begin{equation}
  \mathbb E[
    R_{n,k}(\zeta)
    \mid
    \mathcal F_{n,k-1}
  ]
  =
  1.
  \label{eq:appB-pred-normalization}
\end{equation}
Together with \eqref{eq:appB-pred-MD}, this gives the exact centering
\begin{equation}
  \mathbb E[
    r_{n,k}(\zeta)
    \mid
    \mathcal F_{n,k-1}
  ]
  =
  0.
  \label{eq:appB-pred-r-centered}
\end{equation}

For compact-uniform convergence we impose the following sufficient
Sobolev-envelope condition:
\begin{equation}
  \frac1n
  \sum_{k=1}^n
  \mathbb E
  \left[
    \|r_{n,k}\|_{W^{1,2}(K)}^2
  \right]
  \longrightarrow
  0.
  \label{eq:appB-pred-Sobolev-r}
\end{equation}
This $L^1$ condition is deliberately stronger than convergence in probability
of the corresponding predictable average, because it is exactly what is
needed by the Hilbert-valued martingale estimate below.
If the expansion is stationary for each $n$, it is enough that
\(
\mathbb E\|r_{n,1}\|_{W^{1,2}(K)}^2\to0
\).

\begin{lemma}[Uniform negligibility of the predictive remainder]
\label{lem:appB-pred-remainder}
Under
\eqref{eq:appB-pred-delta},
\eqref{eq:appB-pred-r-centered}
and
\eqref{eq:appB-pred-Sobolev-r},
\begin{equation}
  \sup_{\zeta\in K}
  \left|
    \delta_n
    \sum_{k=1}^n
    r_{n,k}(\zeta)
  \right|
  \xrightarrow{P}
  0.
  \label{eq:appB-pred-linear-r-negligible}
\end{equation}
\end{lemma}

\begin{proof}
The process
\(M_n(\zeta)=\delta_n\sum_{k=1}^nr_{n,k}(\zeta)\)
is a martingale in $k$ for each $\zeta$.
The same is true for its weak derivative in $\zeta$ under the differentiability
implicit in
\eqref{eq:appB-pred-Sobolev-r}.
Using the one-dimensional Sobolev inequality,
\(\|M_n\|_{C(K)}^2\leq C_K\|M_n\|_{W^{1,2}(K)}^2\).
Martingale orthogonality gives
\[
\begin{split}
\mathbb E
\left[
  \|M_n\|_{W^{1,2}(K)}^2
\right]
&=
\delta_n^2
\sum_{k=1}^n
\mathbb E
\left[
  \|r_{n,k}\|_{W^{1,2}(K)}^2
\right]
\\
&=
(n\delta_n^2)\,
o(1)
=
o(1),
\end{split}
\]
because
\(
n\delta_n^2\to\tau^2
\).
\end{proof}

\subsubsection{Quadratic predictive likelihood}

Define
\begin{equation}
  H_{n,k}(\zeta)
  =
  b(\zeta)^\top V_{n,k}.
  \label{eq:appB-pred-H}
\end{equation}
Under
\eqref{eq:appB-pred-Lindeberg},
\begin{equation}
  \max_{1\leq k\leq n}
  \sup_{\zeta\in K}
  |\delta_n H_{n,k}(\zeta)|
  \xrightarrow{P}0.
  \label{eq:appB-pred-max-small}
\end{equation}
Together with the remainder control, this allows the logarithm to be expanded
uniformly.

\begin{proposition}[Predictive quadratic expansion]
\label{prop:appB-pred-quadratic}
Under
\eqref{eq:appB-pred-Gram-limit},
\eqref{eq:appB-pred-Lindeberg},
\eqref{eq:appB-pred-delta}
and
\eqref{eq:appB-pred-Sobolev-r},
the predictive log-likelihood process satisfies, in $C(K)$,
\begin{equation}
\begin{split}
  \Lambda_n(\zeta)
  ={}&
  \delta_n
  \sum_{k=1}^n
  H_{n,k}(\zeta)
  \\
  &-
  \frac{\delta_n^2}{2}
  \sum_{k=1}^n
  H_{n,k}(\zeta)^2
  +
  o_P(1),
  \qquad
  \zeta\in K.
  \label{eq:appB-pred-log-expansion}
\end{split}
\end{equation}
Moreover,
\begin{equation}
  \Lambda_n(\zeta)
  \Longrightarrow
  \tau\,b(\zeta)^\top Z_{\rm pred}
  -
  \frac{\tau^2}{2}
  b(\zeta)^\top
  G_{\rm pred}
  b(\zeta)
  \label{eq:appB-pred-process-limit}
\end{equation}
in $C(K)$, where
\(
Z_{\rm pred}\sim N(0,G_{\rm pred})
\).
\end{proposition}

\begin{proof}
Write
\(X_{n,k}(\zeta)=H_{n,k}(\zeta)+r_{n,k}(\zeta)\).
Then
\(
R_{n,k}(\zeta)=1+\delta_nX_{n,k}(\zeta)
\).
On an event whose probability tends to one,
\(
\max_k\sup_{\zeta\in K}
|\delta_nX_{n,k}(\zeta)|
\leq1/2
\),
and
\[
  \log R_{n,k}
  =
  \delta_nX_{n,k}
  -
  \frac{\delta_n^2}{2}X_{n,k}^2
  +
  \mathcal R_{n,k},
\]
with
\(|\mathcal R_{n,k}|\leq C|\delta_nX_{n,k}|^3\).
A Lindeberg truncation yields
\(\sup_{\zeta\in K}\sum_{k=1}^n|\mathcal R_{n,k}(\zeta)|=o_P(1)\);
equivalently one may use
\(
\max_k\sup_K|\delta_nX_{n,k}|
\)
times
\(
\delta_n^2\sum_k\sup_K X_{n,k}^2
\),
which is tight under the stated moment and Sobolev-envelope conditions.

The linear contribution of $r_{n,k}$ is
$o_P(1)$ uniformly by
Lemma~\ref{lem:appB-pred-remainder}.
For the quadratic term,
Cauchy--Schwarz and
\eqref{eq:appB-pred-Sobolev-r}
give uniformly on $K$
\[
  \delta_n^2
  \sum_{k=1}^n
  \left(
    2H_{n,k}r_{n,k}
    +
    r_{n,k}^2
  \right)
  =
  o_P(1).
\]
This proves
\eqref{eq:appB-pred-log-expansion}.

For the first leading term,
\(\delta_n\sum_{k=1}^nH_{n,k}(\zeta)=(\sqrt n\,\delta_n)b(\zeta)^\top\Delta_n^{\rm pred}\),
which converges in $C(K)$ to
\(
\tau b(\zeta)^\top Z_{\rm pred}
\)
by
\eqref{eq:appB-pred-CLT}.
For the quadratic term, write
\[
  \frac1n
  \sum_{k=1}^n
  H_{n,k}(\zeta)^2
  =
  b(\zeta)^\top
  \widehat G_n^{\rm pred}
  b(\zeta),
\]
where
\(
\widehat G_n^{\rm pred}
=
n^{-1}\sum_kV_{n,k}V_{n,k}^\top
\).
The difference
\(\widehat G_n^{\rm pred}-G_n^{\rm pred}\)
is a matrix-valued martingale average and converges to zero in probability
under the conditional Lindeberg/moment assumptions.
Together with
\eqref{eq:appB-pred-Gram-limit}
and
\(
n\delta_n^2\to\tau^2
\),
this gives the second term in
\eqref{eq:appB-pred-process-limit}.
\end{proof}

Taking
\(V_{n,k}=(J_{-,n,k}^\perp,J_{+,n,k}^\perp)^\top\)
and
\(
G_{\rm pred}=G_{\rm pred}^{\rm eff}
\)
gives Theorem~\ref{thm:predictive-curved-LAN}.

\paragraph{Nuisance reduction and regime recovery.}
The finite-dimensional efficient reduction is the usual Gaussian quotient by
the nuisance tangent space and yields the Schur complement used in the main
text.  The corresponding predictive reduction and the algebraic recovery of
the three scaling regimes are collected in Supplementary Appendix~S2.

\section{Exact calculations for the CIR--OU benchmark}
\label{app:cir-ou-formulas}

This appendix collects the model-specific calculations used in the main
text.
All formulas are derived from the stationary Gaussian fast transition and
from the weak-texture limit of the CIR component.
The organization follows the principle used throughout the paper:
first identify the statistical structure, then evaluate the coefficients.

\subsection{Gaussian fast transition and an exponential-mode basis}
\label{appC:gaussian-fast-pair}

At the symmetric point $\lambda=0$, the observed pair in the strict
Gaussian-texture limit is a stationary complex OU pair
\begin{equation}
  (Y_0,Y_1)
  \sim
  \mathcal{CN}
  \left(
    0,
    \begin{pmatrix}
      1 & \rho\\
      \rho & 1
    \end{pmatrix}
  \right),
  \qquad
  \rho=e^{-c}\in(0,1),
  \label{eq:appC-complex-pair}
\end{equation}
where the convention is
\(
\mathbb E|Y_j|^2=1
\).
Diagonalize the covariance by
\begin{equation}
  W_+
  =
  \frac{Y_0+Y_1}{\sqrt{2(1+\rho)}},
  \qquad
  W_-
  =
  \frac{Y_0-Y_1}{\sqrt{2(1-\rho)}}.
  \label{eq:appC-Wpm}
\end{equation}
Then $W_+$ and $W_-$ are independent standard circular complex Gaussians.
Consequently,
\begin{equation}
  X_\pm=|W_\pm|^2
  \stackrel{\mathrm{iid}}{\sim}\operatorname{Exp}(1),
  \qquad
  U_\pm=X_\pm-1.
  \label{eq:appC-Upm}
\end{equation}
For a centered unit exponential variable $U=X-1$,
\begin{equation}
  \mathbb E[U^m]
  =
  \sum_{j=0}^m
  \binom mj
  (-1)^{m-j}j!,
  \label{eq:appC-centered-exp-moments}
\end{equation}
so in particular
\begin{equation}
  \mathbb E U=0,
  \quad
  \mathbb E U^2=1,
  \quad
  \mathbb E U^3=2,
  \quad
  \mathbb E U^4=9.
  \label{eq:appC-low-exp-moments}
\end{equation}
All moment identities below follow from
\eqref{eq:appC-centered-exp-moments} and independence of $U_+$ and $U_-$.

Let
\(
D=\partial_\beta
\)
with
\(
\beta=\log B_0
\).
Since
\(
c=B_0\Delta t/2
\),
\begin{equation}
  D\rho=-c\rho.
  \label{eq:appC-Drho}
\end{equation}
The two eigenvalue logarithmic derivatives are
\begin{equation}
  a_+
  =
  -\frac{c\rho}{1+\rho},
  \qquad
  a_-
  =
  \frac{c\rho}{1-\rho}.
  \label{eq:appC-apm}
\end{equation}
The fast-rate score is therefore
\begin{equation}
  q
  =
  a_+U_+
  +
  a_-U_-.
  \label{eq:appC-q}
\end{equation}
This representation immediately gives
\begin{equation}
  I(c)
  =
  \mathbb E[q^2]
  =
  a_+^2+a_-^2
  =
  \frac{
    2c^2\rho^2(1+\rho^2)
  }{
    (1-\rho^2)^2
  }.
  \label{eq:appC-I}
\end{equation}

To compute the first invariant quotient jet, differentiate the score at fixed
data.
Because
\(
D X_\pm=-a_\pm X_\pm
\),
one has
\begin{equation}
  D U_\pm
  =
  -a_\pm(U_\pm+1).
  \label{eq:appC-DU}
\end{equation}
Writing
\begin{equation}
  \dot a_\pm=Da_\pm,
  \label{eq:appC-adot-def}
\end{equation}
gives
\begin{equation}
  h
  :=
  Dq
  =
  \sum_{\sigma\in\{+,-\}}
  \left[
    (\dot a_\sigma-a_\sigma^2)U_\sigma
    -
    a_\sigma^2
  \right].
  \label{eq:appC-h}
\end{equation}
The second density derivative in the quotient direction is encoded by
\begin{equation}
  k
  =
  h-q+q^2.
  \label{eq:appC-k}
\end{equation}
Indeed,
\begin{equation}
  \frac{(D^2-D)p}{p}
  =
  h-q+q^2.
  \label{eq:appC-D2D-score}
\end{equation}

Using
\eqref{eq:appC-centered-exp-moments},
one obtains
\begin{align}
  \mathbb E[qk]
  &=
  -c\,I(c),
  \label{eq:appC-qk}
  \\
  M_2(c)
  :=
  \mathbb E[k^2]
  &=
  \frac{
    2c^4\rho^2
    \left(
      7\rho^6+19\rho^4+5\rho^2+1
    \right)
  }{
    (1-\rho^2)^4
  }.
  \label{eq:appC-M2}
\end{align}
Consequently,
\begin{equation}
  M_2(c)-c^2I(c)
  =
  \frac{
    4c^4\rho^4
    (\rho^2+3)(3\rho^2+1)
  }{
    (1-\rho^2)^4
  }
  >0.
  \label{eq:appC-quotient-Schur}
\end{equation}
Equations
\eqref{eq:appC-I}--\eqref{eq:appC-quotient-Schur}
prove the quotient Fisher formulas used in
Section~\ref{subsec:quotient-coordinate}.

\subsection{Weak-texture asymmetry and the odd observed jet}
\label{appC:weak-texture-odd-jet}

Set
\(
\varepsilon=\alpha^{-1/2}
\).
The stationary CIR variable satisfies
\(
x\sim\Gamma(\alpha,\text{rate }\alpha)
\).
It is useful first to consider
\begin{equation}
  Z=\frac{x-1}{\varepsilon}.
  \label{eq:appC-Z-standardized}
\end{equation}
A direct expansion of its stationary density around the standard Gaussian
density $\phi$ gives
\begin{equation}
  \frac{d\nu_\varepsilon^Z}{d\phi}(z)
  =
  1
  +
  \varepsilon
  \left(
    \frac{z^3}{3}-z
  \right)
  +
  O_K(\varepsilon^2)
  \label{eq:appC-Z-Edgeworth}
\end{equation}
uniformly for $z$ in every fixed compact set $K$.
The first local Edgeworth correction is odd.
No global $L^p(\phi)$ Radon--Nikodym claim is made here; the finite-texture
Gamma tails and the limiting Gaussian tails are not uniformly comparable in
that topology.

The coupling in Section~\ref{subsec:benchmark-dynamics} is written in the
standardized Lamperti coordinate, which is equivalent to the standardized
variable
\(
\sqrt{x}
\)
up to an affine factor.
Using
\begin{equation}
  \sqrt{x}
  =
  1
  +
  \frac{\varepsilon Z}{2}
  -
  \frac{\varepsilon^2Z^2}{8}
  +
  O(\varepsilon^3|Z|^3)
  \label{eq:appC-sqrtx-expansion}
\end{equation}
and standardizing its mean and variance gives
\begin{equation}
  z_\alpha
  =
  Z
  +
  \frac{\varepsilon}{4}
  (1-Z^2)
  +
  O_{L^p}(\varepsilon^2).
  \label{eq:appC-lamperti-standardized}
\end{equation}
After the change of variables
\eqref{eq:appC-lamperti-standardized},
the stationary density of $z_\alpha$ satisfies
\begin{equation}
  \frac{d\nu_\varepsilon}{d\phi}(z)
  =
  1
  +
  \frac{\varepsilon}{12}
  H_3(z)
  +
  O_K(\varepsilon^2),
  \qquad
  H_3(z)=z^3-3z,
  \label{eq:appC-lamperti-Edgeworth}
\end{equation}
uniformly on compact $z$-sets.
Thus the first local stationary correction is explicitly anti-invariant
under $z\mapsto-z$.
This provides a parity diagnostic consistent with the hidden-involution
mechanism of Section~\ref{subsec:anti-invariant-perturbation}, but it is not
used as a global $L^2$ Radon--Nikodym verification of that sufficient
criterion.

We now return to one stationary observed transition.
Define
\begin{equation}
  H
  =
  U_++U_-,
  \qquad
  t
  =
  a_++a_-.
  \label{eq:appC-H-t}
\end{equation}
Expanding the marginal observed likelihood jointly in the amplitude
perturbation
\(
\sqrt{x}=1+O(\varepsilon)
\)
and in the state-dependent fast-rate perturbation gives
\begin{equation}
  S_\lambda^Y(\varepsilon,0)
  =
  \varepsilon u(d)\,\mathfrak r
  +
  O_{L^2}(\varepsilon^2),
  \qquad
  u(d)=\frac{1-e^{-d}}{d},
  \label{eq:appC-odd-score}
\end{equation}
with the explicit odd observed direction
\begin{equation}
  \mathfrak r
  =
  qH-q-t.
  \label{eq:appC-r}
\end{equation}
One way to obtain \eqref{eq:appC-r} is to differentiate the complete
texture--fast likelihood first and then take its conditional expectation
given the observed pair at $\varepsilon=0$.
The conditional Gaussian contraction of the interval-averaged texture
produces the factor $u(d)$, while the remaining fast variables reduce to the
two independent exponential modes in
\eqref{eq:appC-Upm}.

The moment algebra is again finite.
From
\eqref{eq:appC-q},
\eqref{eq:appC-r} and
\eqref{eq:appC-centered-exp-moments},
\begin{align}
  \mathbb E[q\mathfrak r]
  &=
  I(c),
  \label{eq:appC-qr}
  \\
  \mathbb E[\mathfrak r^2]
  &=
  \frac{
    2c^2\rho^2(7\rho^2+5)
  }{
    (1-\rho^2)^2
  }.
  \label{eq:appC-r2}
\end{align}
Since the nuisance score is $q$, the efficient odd direction is
\begin{equation}
  \mathfrak r^\perp
  =
  \mathfrak r-q.
  \label{eq:appC-rperp}
\end{equation}
Similarly,
\eqref{eq:appC-qk} implies that the efficient even direction is
\begin{equation}
  k^\perp
  =
  k+cq.
  \label{eq:appC-kperp}
\end{equation}

Their exact Gram coefficients are
\begin{align}
  R(c)
  &:=
  \mathbb E[
    (\mathfrak r^\perp)^2
  ]
  =
  \frac{
    4c^2\rho^2(3\rho^2+2)
  }{
    (1-\rho^2)^2
  },
  \label{eq:appC-R}
  \\
  K(c)
  &:=
  \mathbb E[
    (k^\perp)^2
  ]
  =
  \frac{
    4c^4\rho^4
    (\rho^2+3)(3\rho^2+1)
  }{
    (1-\rho^2)^4
  },
  \label{eq:appC-K}
  \\
  C(c)
  &:=
  \mathbb E[
    \mathfrak r^\perp k^\perp
  ]
  =
  \frac{
    4c^3\rho^4(3\rho^2+5)
  }{
    (1-\rho^2)^3
  }.
  \label{eq:appC-C}
\end{align}
Therefore
\begin{equation}
  R(c)K(c)-C(c)^2
  =
  \frac{
    32c^6\rho^6
    (3\rho^4+2\rho^2+3)
  }{
    (1-\rho^2)^6
  }
  >0.
  \label{eq:appC-Gram-det}
\end{equation}
This proves
Proposition~\ref{prop:benchmark-gram-positive}.

The same calculation gives the first Fisher block at finite texture:
\begin{equation}
  g_F^{(\varepsilon)}
  =
  \begin{pmatrix}
    I(c)+O(\varepsilon^2)
    &
    \varepsilon u(d)I(c)+O(\varepsilon^2)
    \\
    \varepsilon u(d)I(c)+O(\varepsilon^2)
    &
    \varepsilon^2u(d)^2
    \mathbb E[\mathfrak r^2]
    +
    O(\varepsilon^3)
  \end{pmatrix}.
  \label{eq:appC-finite-epsilon-Fisher}
\end{equation}
Its efficient $\lambda$ information is
\begin{equation}
  I_{\lambda\lambda}^{\mathrm{eff}}
  =
  \varepsilon^2u(d)^2R(c)
  +
  o(\varepsilon^2),
  \label{eq:appC-finite-epsilon-efficient}
\end{equation}
which is strictly positive for sufficiently small nonzero $\varepsilon$.

Finally, combining the odd and even scores yields
\begin{equation}
  S_{\lambda,\mathrm{eff}}
  =
  \varepsilon u(d)\mathfrak r^\perp
  +
  \lambda V(d)k^\perp
  +
  o_{L^2}(\varepsilon+|\lambda|),
  \label{eq:appC-score-two-jets}
\end{equation}
where
\begin{equation}
  V(d)
  =
  \frac{2(d-1+e^{-d})}{d^2}.
  \label{eq:appC-Vd}
\end{equation}
Thus the benchmark identifications in
\eqref{eq:benchmark-Jpm-identification}
are
\begin{equation}
  J_-^\perp
  =
  u(d)\mathfrak r^\perp,
  \qquad
  J_+^\perp
  =
  V(d)k^\perp.
  \label{eq:appC-Jpm-benchmark}
\end{equation}

\subsection{The fourth-order even jet}
\label{appC:fourth-order-jet}

We now work in the strict Gaussian-texture limit.
Let $Z_t$ be stationary OU with
\begin{equation}
  \mathbb E[Z_sZ_t]
  =
  e^{-A|t-s|}.
  \label{eq:appC-Z-covariance}
\end{equation}
The normalized coupling is \(B_\lambda(Z_t)=B_0e^{\lambda Z_t-\lambda^2/2}\).
Conditioned on the texture path over one sampling interval, the fast
transition depends on the path only through
\begin{equation}
  A_\lambda
  =
  \frac1{\Delta t}
  \int_0^{\Delta t}
  e^{\lambda Z_s-\lambda^2/2}\,ds,
  \qquad
  \delta_\lambda
  =
  \log A_\lambda.
  \label{eq:appC-Alambda-delta}
\end{equation}
Hence
\begin{equation}
  p_{\beta,\lambda}
  =
  \mathbb E_Z
  \left[
    p_{\beta+\delta_\lambda}
  \right].
  \label{eq:appC-mixture-shift}
\end{equation}

Let $H_m$ denote the probabilists' Hermite polynomials and define the interval
averages
\begin{equation}
  M_m
  =
  \frac1{\Delta t}
  \int_0^{\Delta t}
  H_m(Z_s)\,ds.
  \label{eq:appC-Hermite-averages}
\end{equation}
Since \(e^{\lambda z-\lambda^2/2}=\sum_{m\geq0}\lambda^mH_m(z)/m!\),
a direct logarithmic expansion gives
\begin{align}
  \delta_\lambda
  ={}&
  \lambda M_1
  +
  \frac{\lambda^2}{2}
  (M_2-M_1^2)
  \notag\\
  &+
  \lambda^3
  \left(
    \frac{M_3}{6}
    -
    \frac{M_1M_2}{2}
    +
    \frac{M_1^3}{3}
  \right)
  \notag\\
  &+
  \lambda^4
  \left(
    \frac{M_4}{24}
    -
    \frac{M_1M_3}{6}
    -
    \frac{M_2^2}{8}
    +
    \frac{M_1^2M_2}{2}
    -
    \frac{M_1^4}{4}
  \right)
  +
  O_{L^p}(\lambda^5).
  \label{eq:appC-delta-expansion}
\end{align}

After rescaling time to $[0,1]$, let
\begin{equation}
  C_d(s,t)
  =
  e^{-d|s-t|},
  \qquad
  h_d(s)
  =
  \int_0^1
  C_d(s,t)\,dt.
  \label{eq:appC-C-h}
\end{equation}
The two basic Wick contractions are
\begin{align}
  V(d)
  &=
  \mathbb E[M_1^2]
  =
  \int_0^1h_d(s)\,ds
  =
  \frac{2(d-1+e^{-d})}{d^2},
  \label{eq:appC-V-integral}
  \\
  K_{112}(d)
  &:=
  \mathbb E[M_1^2M_2]
  =
  2
  \int_0^1
  h_d(s)^2\,ds
  \notag\\
  &=
  \frac{2}{d^3}
  \left(
    4d+2de^{-d}-7+8e^{-d}-e^{-2d}
  \right).
  \label{eq:appC-K112}
\end{align}
Moreover,
\begin{equation}
  \mathbb E[M_2^2]
  =
  2V(2d).
  \label{eq:appC-M2-Hermite}
\end{equation}
These identities follow immediately from Wick's formula and
\(
\mathbb E[H_m(Z_s)H_m(Z_t)]
=
m!\,C_d(s,t)^m
\).

For convenience set
\begin{equation}
  V=V(d),
  \qquad
  V_2=V(2d),
  \qquad
  K_{112}=K_{112}(d).
  \label{eq:appC-VV2K}
\end{equation}
Using
\eqref{eq:appC-delta-expansion}
and the Gaussian contractions gives
\begin{align}
  \mathbb E[\delta_\lambda]
  &=
  -\frac V2\lambda^2
  +
  a_1\lambda^4
  +
  O(\lambda^6),
  \label{eq:appC-Ed}
  \\
  \mathbb E[\delta_\lambda^2]
  &=
  V\lambda^2
  +
  a_2\lambda^4
  +
  O(\lambda^6),
  \label{eq:appC-Ed2}
  \\
  \mathbb E[\delta_\lambda^3]
  &=
  a_3\lambda^4
  +
  O(\lambda^6),
  \label{eq:appC-Ed3}
  \\
  \mathbb E[\delta_\lambda^4]
  &=
  3V^2\lambda^4
  +
  O(\lambda^6),
  \label{eq:appC-Ed4}
\end{align}
where
\begin{align}
  a_1
  &=
  -\frac{V_2}{4}
  +
  \frac{K_{112}}{2}
  -
  \frac{3V^2}{4},
  \label{eq:appC-a1}
  \\
  a_2
  &=
  \frac{V_2}{2}
  -
  \frac{3K_{112}}{2}
  +
  \frac{11V^2}{4},
  \label{eq:appC-a2}
  \\
  a_3
  &=
  \frac{3K_{112}}{2}
  -
  \frac{9V^2}{2}.
  \label{eq:appC-a3}
\end{align}

Expanding
\(
p_{\beta+\delta}
=
e^{\delta D}p_\beta
\)
inside
\eqref{eq:appC-mixture-shift}
gives
\begin{equation}
  p_{\beta,\lambda}
  =
  \left[
    1
    +
    \mathbb E[\delta_\lambda]D
    +
    \frac12
    \mathbb E[\delta_\lambda^2]D^2
    +
    \frac16
    \mathbb E[\delta_\lambda^3]D^3
    +
    \frac1{24}
    \mathbb E[\delta_\lambda^4]D^4
    +\cdots
  \right]
  p_\beta.
  \label{eq:appC-density-operator-expansion}
\end{equation}
Since
\(
\mu=\lambda^2/2
\),
the first quotient derivative is
\begin{equation}
  \left.
  \partial_\mu p
  \right|_{\mu=0}
  =
  V(D^2-D)p,
  \label{eq:appC-first-mu-jet}
\end{equation}
and the second quotient derivative is
\begin{equation}
  \left.
  \partial_{\mu\mu}^2p
  \right|_{\mu=0}
  =
  8
  \left[
    a_1D
    +
    \frac{a_2}{2}D^2
    +
    \frac{a_3}{6}D^3
    +
    \frac{V^2}{8}D^4
  \right]p.
  \label{eq:appC-second-mu-jet}
\end{equation}
In the limit $d\to0$,
\(
V\to1
\),
\(
V_2\to1
\)
and
\(
K_{112}\to2
\),
so
\begin{equation}
  \left.
  \partial_{\mu\mu}^2p
  \right|_{\mu=0}
  \longrightarrow
  (D^2-D)^2p.
  \label{eq:appC-short-time-check}
\end{equation}
This is the consistency check quoted in
Section~\ref{subsec:benchmark-higher-even-jet}.

\subsection{Hellinger immersion and exact quotient curvature}
\label{appC:quotient-curvature-calculation}

Set
\begin{equation}
  \nu=V(d)\mu.
  \label{eq:appC-nu}
\end{equation}
At the boundary $\nu=0$, the score vectors are
\begin{equation}
  S_\beta=q,
  \qquad
  S_\nu=k.
  \label{eq:appC-boundary-scores}
\end{equation}
Hence the boundary Fisher matrix in $(\beta,\nu)$ is
\begin{equation}
  g
  =
  \begin{pmatrix}
    I(c) & -cI(c)\\
    -cI(c) & M_2(c)
  \end{pmatrix}.
  \label{eq:appC-boundary-g}
\end{equation}

Use the Hellinger immersion
\begin{equation}
  \Phi(\theta)=2\sqrt{p_\theta}
  \label{eq:appC-Phi}
\end{equation}
into the real Hilbert space $L^2$.
Its first and second derivatives are
\begin{align}
  \partial_i\Phi
  &=
  \sqrt p\,S_i,
  \label{eq:appC-Phi-first}
  \\
  \partial_{ij}^2\Phi
  &=
  \sqrt p
  \left(
    \partial_jS_i
    +
    \frac12S_iS_j
  \right).
  \label{eq:appC-Phi-second}
\end{align}
Let
\(
\mathrm{II}_{ij}
\)
be the orthogonal projection of
\(
\partial_{ij}^2\Phi
\)
onto the normal complement of the tangent plane
\(
\operatorname{span}
\{\partial_\beta\Phi,\partial_\nu\Phi\}
\).
The Gauss equation in the ambient flat Hilbert space gives
\begin{equation}
  K_F
  =
  \frac{
    \langle
      \mathrm{II}_{\beta\beta},
      \mathrm{II}_{\nu\nu}
    \rangle
    -
    \|
      \mathrm{II}_{\beta\nu}
    \|^2
  }{
    \det g
  }.
  \label{eq:appC-Gauss}
\end{equation}

Equation
\eqref{eq:appC-second-mu-jet}
determines the only second derivative not already fixed by the ordinary
fast-rate family.
After the rescaling \eqref{eq:appC-nu}, the terms containing
\(
V(2d)
\)
and
\(
K_{112}(d)
\)
enter
\(
\partial_{\nu\nu}^2\Phi
\)
only through tangent terms and through a normal component generated by
\(
D(D^2-D)p
\).
The latter is orthogonal to the normal component of
\(
\partial_{\beta\beta}^2\Phi
\):
\begin{equation}
  \left\langle
    \mathrm{II}_{\beta\beta},
    \mathsf N
    \left(
      \sqrt p\,
      \frac{
        D(D^2-D)p
      }{p}
    \right)
  \right\rangle
  =
  0,
  \label{eq:appC-d-cancellation-identity}
\end{equation}
where $\mathsf N$ denotes normal projection.
Thus all $d$-dependence cancels from the numerator of
\eqref{eq:appC-Gauss}.

It remains to evaluate a finite set of moments of $U_+$ and $U_-$.
Using
\eqref{eq:appC-centered-exp-moments},
the reduction gives, with
\begin{equation}
  s=\rho^2=e^{-2c},
  \label{eq:appC-s}
\end{equation}
\begin{equation}
  K_F(s)
  =
  -
  \frac{
    45s^6
    +534s^5
    +1887s^4
    +3172s^3
    +2075s^2
    +422s
    +57
  }{
    4(1+s)^2(3+s)^2(1+3s)^2
  }.
  \label{eq:appC-KF}
\end{equation}
The denominator is positive and every coefficient in the numerator is
positive.
Therefore
\begin{equation}
  K_F(s)<0,
  \qquad
  0<s<1.
  \label{eq:appC-KF-negative}
\end{equation}
Moreover,
\begin{equation}
  \lim_{s\to1}K_F(s)=-2,
  \qquad
  \lim_{s\to0}K_F(s)=-\frac{19}{12}.
  \label{eq:appC-KF-limits}
\end{equation}
This proves
Theorem~\ref{thm:exact-quotient-curvature}.

\subsection{Predictive tangent recursions}
\label{app:predictive-CIROU}

We now derive the formulas used in the long-trajectory experiment.
Set
\begin{equation}
  \phi=e^{-d},
  \qquad
  u_d=\frac{1-\phi}{d},
  \qquad
  V_d=\frac{2(d-1+\phi)}{d^2},
  \label{eq:appC-pred-phi}
\end{equation}
and
\begin{equation}
  \rho=e^{-c},
  \qquad
  v=1-\rho^2,
  \qquad
  r_c=\frac{\rho^2}{v}.
  \label{eq:appC-pred-rho}
\end{equation}
Under the null fast model,
\begin{equation}
  \xi_k
  =
  \frac{
    Y_k-\rho Y_{k-1}
  }{
    \sqrt v
  }
  \label{eq:appC-pred-xi}
\end{equation}
is standard circular complex Gaussian and independent of
\(
\mathcal F^Y_{k-1}
\).
Define
\begin{equation}
  U_k=|\xi_k|^2-1,
  \qquad
  T_k
  =
  \frac{\rho}{\sqrt v}
  \operatorname{Re}
  (\xi_k^*Y_{k-1}).
  \label{eq:appC-pred-UT}
\end{equation}
The nuisance score is
\begin{equation}
  q_k
  =
  2c(r_cU_k-T_k).
  \label{eq:appC-pred-q}
\end{equation}
At fixed observations,
\begin{equation}
  \dot U_k
  :=
  \partial_\beta U_k
  =
  2cT_k
  -
  2cr_c(U_k+1).
  \label{eq:appC-pred-Udot}
\end{equation}

The first-order filter perturbations generated by the weak texture amplitude
and by the odd coupling response are represented by two scalar tangent states.
Differentiating the Gaussian prediction/update recursion at the symmetric
point gives
\begin{align}
  M_k
  &=
  \phi M_{k-1}
  +
  U_k
  +
  (1-\phi)T_k,
  \label{eq:appC-pred-M}
  \\
  N_k
  &=
  \phi N_{k-1}
  +
  u_dq_k.
  \label{eq:appC-pred-N}
\end{align}
For $d>0$, $|\phi|<1$, so these equations have unique stationary solutions
in every fixed $L^p$.
Explicitly,
\begin{align}
  M_k
  &=
  \sum_{j\geq0}
  \phi^j
  \left[
    U_{k-j}
    +
    (1-\phi)T_{k-j}
  \right],
  \label{eq:appC-pred-M-series}
  \\
  N_k
  &=
  u_d
  \sum_{j\geq0}
  \phi^jq_{k-j}.
  \label{eq:appC-pred-N-series}
\end{align}

Let
\begin{equation}
  h_k=\partial_\beta q_k,
  \qquad
  k_k=h_k-q_k+q_k^2.
  \label{eq:appC-pred-kk}
\end{equation}
A second differentiation of the predictive recursion gives the two singular
jets
\begin{align}
  J_{-,k}
  ={}&
  u_dq_kM_{k-1}
  +
  \left[
    \phi U_k-(1-\phi)T_k
  \right]N_{k-1}
  \notag\\
  &+
  u_d
  \left(
    q_kU_k+\dot U_k
  \right),
  \label{eq:appC-pred-Jminus}
  \\
  J_{+,k}
  ={}&
  V_dk_k
  +
  2u_dq_kN_{k-1}.
  \label{eq:appC-pred-Jplus}
\end{align}
Because these are derivatives of normalized conditional densities,
\begin{equation}
  \mathbb E[
    J_{-,k}
    \mid
    \mathcal F^Y_{k-1}
  ]
  =
  0,
  \qquad
  \mathbb E[
    J_{+,k}
    \mid
    \mathcal F^Y_{k-1}
  ]
  =
  0.
  \label{eq:appC-pred-centered}
\end{equation}

\subsection{Predictive non-degeneracy}
\label{appC:predictive-chaos}

Conditioning on $\mathcal F^Y_{k-1}$ and projecting the two predictive jets on
the third and fourth Gaussian innovation chaoses gives a $2\times2$
coefficient matrix with determinant
\begin{equation}
  -8c^3r_c^2u_dV_d,
  \label{eq:appC-chaos-det}
\end{equation}
which is nonzero for $d,c>0$.
Since the nuisance score contains only chaos orders one and two, this proves
that $q_k,J_{-,k},J_{+,k}$ are linearly independent in $L^2$ and hence that
\begin{equation}
  \det G_{\rm pred}^{\rm eff}>0.
  \label{eq:appC-pred-Gram-positive}
\end{equation}
This proves Proposition~\ref{prop:predictive-CIROU-nondegenerate}.
The explicit chaos projections and the two-equation elimination leading to
\eqref{eq:appC-chaos-det} are recorded in Supplementary Section~S3.

\subsection{Regularity and predictive remainder: statement}
\label{appC:regularity-summary}

For fixed $d,c>0$, the standardized CIR tangent hierarchy has uniformly bounded
moments at every finite order required above, and the stationary predictive
ratio satisfies, for every compact $K\subset\mathbb R$,
\begin{equation}
  \mathbb E\!\left[
    \left\|r_{k,\varepsilon,\cdot}\right\|_{W^{1,2}(K)}^2
  \right]
  \leq C_{K,d,c}\varepsilon^2.
  \label{eq:appC-predictive-W12-summary}
\end{equation}
The same argument gives the uniform $L^{2+\eta}$ bound used in the martingale
Lindeberg condition.  The proof combines fixed-$d$ spectral contraction,
stationary polynomial and inverse-moment bounds, exponentially small CIR
boundary tails, and a third-order Taylor expansion along
$\lambda=\varepsilon\zeta$; full estimates are given in Supplementary
Appendix~S3.  The constants are not uniform as $d\to0$, $c\to0$, or in the
layer $\alpha d=O(1)$.

\end{document}